\documentclass[11pt,a4paper]{amsart}

\usepackage[utf8]{inputenc}
\usepackage[T1]{fontenc}      
\usepackage{geometry}         
\usepackage[french, english]{babel}  
\usepackage{amsmath}
\usepackage{amsfonts}
\usepackage{amsthm}

\usepackage{amscd}
\usepackage{hyperref}
\usepackage{xcolor}
\usepackage{enumitem}

\usepackage{graphicx}
\usepackage{caption}
\usepackage{subcaption}

\usepackage{pstricks,pst-plot}

   \usepackage{amssymb}
   \usepackage{verbatim}
   \newif\ifpdf
   \ifpdf
     \usepackage[pdftex]{graphicx}
     \usepackage[pdftex]{hyperref}
   \else
     \usepackage{graphicx}
   \fi

\usepackage[all]{xy}

\usepackage{enumitem}

\newcommand{\R}{\mathbb{R}}

\newcommand{\N}{\mathbb{N}}
\newcommand{\Z}{\mathbb{Z}}
\newcommand{\C}{\mathbb{C}}

\newcommand{\D}{\mathbb{D}}

\newcommand{\teich}{\mathrm{Teich}}

\newcommand{\rcal}{\mathcal{R}}
\newcommand{\lam}{\lambda}

\newcommand{\horn}{H}

\newtheorem*{mainthm}{Main Theorem}
\newtheorem{thm}{Theorem}[section]
\newtheorem{prop}[thm]{Proposition}
\newtheorem{coro}[thm]{Corollary}
\newtheorem{defi}[thm]{Definition}
\newtheorem{lem}[thm]{Lemma}

\newtheorem{qst}[thm]{Question}
\newtheorem{rem}[thm]{Remark}

\newtheorem{ex}[thm]{Example}

\newcommand{\rs}{\mathbb{P}^1}

\newcommand{\limn}{\lim_{n \rightarrow \infty}}
\newcommand{\id}{\mathrm{Id}}

\newcommand{\mcal}{\mathcal{M}}

\newcommand{\bcal}{\mathcal{B}}
\newcommand{\ptwo}{\mathbb{P}^2}
\newcommand{\pk}{\mathbb{P}^k}

\newcommand{\lcal}{\mathcal{L}}

\newcommand{\re}{\mathrm{Re}}
\newcommand{\im}{\mathrm{Im}}

\newcommand{\proj}{\mathrm{proj}}

\numberwithin{equation}{section}

\newcommand{\eps}{\epsilon}

\author{Matthieu Astorg}
\thanks{The first author was partially supported by the ANR grant Fatou ANR-17-CE40-0002-01. }

\author{Luka Boc Thaler}
\thanks{The second author was supported by the research program J1-3005 from ARRS, Republic of Slovenia}

\address{M. Astorg: Institut Denis Poisson, Collegium Sciences et Techniques, Université d'Orléans
	Rue de Chartres B.P. 6759
	45067 Orléans cedex 2 France.}
\email{matthieu.astorg@univ-orleans.fr}

\address{L. Boc Thaler: Faculty of Education, University of Ljubljana, SI--1000 Ljubljana, Slovenia. Institute of Mathematics, Physics and Mechanics, Jadranska 19, 1000 Ljubljana, Slovenia.} \email{luka.boc@pef.uni-lj.si}

\title{Dynamics of skew-products tangent to the identity}

\begin{document}

\begin{abstract}
	We study the local dynamics of generic skew-products tangent to the identity, i.e. maps of the form $P(z,w)=(p(z), q(z,w))$ with $dP_0=\id$. More precisely, we focus on  maps with non-degenerate second differential at the origin; such maps have local normal form  $P(z,w)=(z-z^2+O(z^3),w+w^2+bz^2+O(\|(z,w)\|^3))$. 
	We prove the existence of parabolic domains, and prove that inside these parabolic domains the orbits converge non-tangentially  if and only if $b \in (\frac{1}{4},+\infty)$. Furthermore, we prove the existence of a type of parabolic implosion, in which the renormalization limits are different from previously known cases. This has a number of consequences: under a diophantine condition on coefficients of $P$, we prove the existence of wandering domains with rank 1 limit maps. We also give explicit examples of quadratic skew-products with countably many grand orbits of wandering domains, and we give an explicit example of a skew-product map with a Fatou component exhibiting historic behaviour. Finally, we construct various topological invariants,
	which allow us to answer a question of Abate.
\end{abstract}

\maketitle

\section{Introduction}
Skew-products are holomorphic self-maps of $\C^2$ of the form 
$$P(z,w)=(p(z), q(z,w)).$$ 
An important feature of these maps is that they preserve the set of vertical lines in $\C^2$. This means that we can view the restriction of $P^n$ to a line $\{z\}\times\C$ as the composition of $n$ entire functions on $\C$, which allows techniques from one-dimensional complex dynamics to be applied. The dynamics of skew-products is therefore in some ways reminiscent of the dynamics of one-variable maps; however, in recent years, several important results have shown that these maps have rich and interesting dynamics, 
see \cite{Jonsson, PetersSmit, PetersVivas, vivas2018local}. 
For example, in \cite{ABDPR}, it was shown that there exists polynomial skew-products, i.e. $P$ is a polynomial map, with \emph{wandering Fatou components},  a dynamical phenomenon that is known not to occur for polynomial maps in one complex dimension.
The proof of the main result in that paper involves  the adaptation  of  \emph{parabolic implosion} to 
the skew-product setting (see also \cite{bedford2012parabolic, bianchi2019parabolic, astorg2019wandering}
for further results on parabolic implosion in several complex variables). 
Polynomial skew-products were also used in \cite{dujardin2016non} and \cite{taflin_blender} to construct \emph{robust} bifurcations, i.e. open sets contained in the bifurcation locus of the family of endomorphisms of $\ptwo$ of given algebraic degree $d \geq 2$. 

\medskip 

Given a germ of a holomorphic self-map  $P$ of $\mathbb{C}^2$ that fixes the origin, we say that $P$ is \emph{tangent to the identity} if it is of the form $P = \mathrm{Id} + P_k(z,w) + O(\|(z,w)\|^{k+1})$, where $k\geq2$ and $P_k: \C^2 \to \C^2$ is a non-trivial homogeneous polynomial map of degree $k$.  The study of local dynamics of germs  tangent to the identity has received significant attention over the last decades. For general germs of $(\C^2,0)$ tangent to the identity, a complete description of the dynamics on a full neighborhood of the origin is for now far out of reach. Much effort has been instead devoted to investigating the existence of invariant manifolds or invariant formal curves on which the dynamics converges to the origin (see e.g. \cite{hakim1998analytic, abate2001residual}, and more recently \cite{lopez2020stable, lopez2021stable}).

\medskip

In this paper we investigate the local dynamics of skew-products $P$ which are tangent to the identity  and have a  non-degenerate second order differential at the origin.  
Up to conjugacy by a linear automorphism of $\C^2$, such maps have the form

$$
P: (z,w)\mapsto \left(z-z^2+O(z^3),  w+w^2+bz^2+O(\|(z,w)\|^3)\right),
$$

and after a second conjugacy by an automorphism of $\C^2$ of the form 
$$(z,w) \mapsto (z, e^{Az}w+Bz^2),$$ we may finally assume that $P$ is of the form $P(z,w)=(p(z), q(z,w))$
with
\begin{equation}\label{map1}
\left\{
	\begin{array}{l}
		p(z):= z-z^2+a z^3+O(z^4)       \\
     q(z,w):=  w+w^2+bz^2+b_{0,3} w^3+b_{3,0} z^3+O(\|(z,w)\|^4)
	\end{array}
\right.
\end{equation}
where $a,b,b_{0,3},b_{3,0} \in \C$. 

Throughout this paper, we will be using the notation $q_z(w):=q(z,w)$ (in particular, $q_0=q(0,\cdot)$).
\medskip

 A study of the local dynamics of skew-products in the case $b=0$ in \eqref{map1}
has  been undertaken in \cite{vivas2018local}, where a full description of the dynamics on a neighborhood of a parabolic fixed point at the origin was achieved.
However, most of the difficulty and richness of the dynamics (including the phenomenon of parabolic implosion and the existence of wandering domains)  comes precisely from this term $bz^2$.

In fact, although maps of the form \eqref{map1} are generic among polynomial skew-products which are tangent to the identity (after analytic conjugacy), we will see that they have considerably complicated local dynamics. We see the investigation of those maps 
\eqref{map1} and the results of this paper as a first step (generic case) towards 
the systematic analysis of the local dynamics of all polynomial skew-products which are tangent to the identity. 

\subsection{Parabolic domains and parabolic implosion} 

\begin{defi} Let P be a holomorphic self-map of $\C^2$ with a parabolic fixed point at the origin. A \emph{parabolic domain} of $P$ is a maximal connected domain $U \subset \C^2$ such that the origin is contained in the boundary of $U$ and the iterates $P^n_{|U}$ converges locally uniformly on $U$ to the origin. 
\end{defi}

We begin by discussing the existence of parabolic domains for maps of the form \eqref{map1}, which depends only on $b$:

\begin{thm}\label{th:parbas}
	Let $P$ be a map of the form \eqref{map1}. Then
	\begin{enumerate}
		\item If $b \in (\frac{1}{4},+\infty)$, the map $P$ has at least  two invariant
		parabolic domains, in which orbits converge non-tangentially to the origin.
		\item If $b\in\mathbb{C} \backslash (\frac{1}{4},+\infty)$, the map $P$ has an invariant parabolic domain, in which each point is attracted to the origin along trajectories tangent to one of its non-degenerate characteristic directions.
	\end{enumerate}
\end{thm}

The main novelty here lies in the first statement of this theorem, while the second statement can be deduced  from results of Hakim and Vivas. 
Invariant parabolic domains in which points converge non-tangentially to the origin are also sometimes called \emph{spiral domains} (see the beginning of Section \ref{sec:parbas} for a precise definition). Such domains were first constructed by Rivi in her thesis \cite[Proposition 4.4.4]{rivi}. In \cite{rong}, Rong gave sufficient conditions for the existence of spiral domains for some class of maps tangent to the identity (see \cite[Theorem 1.4]{rong}). However, his result does not apply to maps of the form \eqref{map1}.

\medskip

From now on we will assume that $b>\frac{1}{4}$, and we introduce the following notations:

\begin{equation}\label{eq:c} 
	c:=\frac{\sqrt{4b-1}}{2}, \quad \alpha_0:=e^{\pi/c}, \quad \beta_0:=(b_{0,3}-a)(\alpha_0-1) .
\end{equation}

Observe that since $b>\frac{1}{4}$, we have $c>0$ and $\alpha_0>1$.

In what follows we will see that in the case $b>\frac{1}{4}$ and $\beta_0\in\R$, there is   parabolic implosion, which has many interesting dynamical consequences. 

\begin{defi} Let $P$ be of the form \eqref{map1}, and $\alpha,\sigma \in \C$. Its generalized Lavaurs map  of phase $\sigma$ and parameter $\alpha$ is defined as 
	\begin{equation}\label{glavaurs}
	\mathcal{L}(\alpha,\sigma; z,w):=(\phi^o_{q_0})^{-1}\left(\alpha\phi^{\iota}_{q_0}(w)+(1-\alpha)\phi^{\iota}_p(z)+\sigma\right),
	\end{equation}
	where $\phi_p^\iota$ is the incoming Fatou coordinate of $p$ and $\phi_{q_0}^{\iota/o}$ are the incoming and outgoing Fatou coordinates of $q_0$ respectively.
	\end{defi}

The definitions and basic properties of Fatou coordinates are recalled in Subsection \ref{subsec:fatouc}; for some more background on Fatou coordinates, Lavaurs maps
and horn maps, see e.g. the Appendix of \cite{ABDPR}.
The generalized Lavaurs map is defined for $(z,w) \in \bcal_p \times \bcal_{q_0}$, where  $\bcal_p$ and $\bcal_{q_0}$ are basins of a parabolic fixed point at the origin for $p$ and $q_0$ respectively.
If $\alpha=1$, then the map $w \mapsto \lcal(\alpha,\sigma; z,w)$ does not depend on $z$ and coincides with the classical Lavaurs map of phase $\sigma$ of the one-variable polynomial $q_0$.
Moreover, generalized Lavaurs maps satisfy the following functional relation:

\begin{equation}\label{invariance}
	\lcal(\alpha,\sigma; p(z), q_0(w))= q_0 \circ \lcal(\alpha,\sigma; z,w) = \lcal(\alpha,\sigma+1; z,w)
\end{equation}
for all $(z,w) \in \bcal_p \times \bcal_{q_0}$.

\medskip

\begin{defi} Given real numbers $\alpha>1$ and $\beta\in\mathbb{R}$ we say that a strictly increasing sequence of positive integers $(n_k)_{k\geq 0}$ is \emph{$(\alpha,\beta)$-admissible} if and only if its \emph{phase sequence} $(\sigma_k)_{k\geq 0}$, defined by $\sigma_k:= n_{k+1}-\alpha n_k-\beta\ln{n_k}$, is bounded. In the case where $\beta=0$, we will simply call such a sequence \emph{$\alpha$-admissible}.
\end{defi}

Observe that for any $\alpha>1$ and $\beta \in \R$, there always exists  $(\alpha,\beta)$-admissible sequences:
it suffices to define inductively $n_{k+1}:=\lfloor \alpha n_k + \beta \ln n_k \rfloor$ and take $n_0 \in \N$ large enough,
where $\lfloor \cdot \rfloor$ denotes the floor function. 
However, describing the phase sequence is in general a difficult problem; for instance, even in the particular case of the $\frac{3}{2}$-admissible sequences $n_{k+1}=\lfloor \frac{3}{2} n_k \rfloor$, the phase sequence is not fully understood (see \cite{dubickas2009integer}). 
An interesting question is the existence of 
$(\alpha,\beta)$-admissible sequences with \emph{converging} phase sequence, which will be discussed in detail below.

\medskip

The following is the main technical result of this paper.
\begin{mainthm}\label{main}  Let $P$ be a map of the form \eqref{map1}. Let $\alpha_0,\beta_0$ be as   in \eqref{eq:c}, and assume that $b>\frac{1}{4}$ and $\beta_0 \in \R$. Let $(n_k)_{k \geq 0}$ be an $(\alpha_0,\beta_0)$-admissible sequence and let $(\sigma_k)_{k \geq 0}$ denote its phase sequence. Then
	\medskip
	
		\begin{equation*}
	P^{n_{k+1}-n_k}(p^{n_k}(z),w) = \left(0, \lcal(\alpha_0, \Gamma + \sigma_k; z,w) \right) + o(1) \qquad (\text{ as } k \to +\infty)
	\end{equation*}
	\medskip
	
	with uniform convergence on compacts in  $\bcal_p \times \bcal_{q_0}$
	and where $\Gamma$ is a  constant depending only on $a,b,b_{0,3},b_{3,0}$ (see \eqref{eq:Gamma} for its explicit expression).
\end{mainthm}

See Remark \ref{rem:general} for a discussion of the case where $b>\frac{1}{4}$ and  $\beta_0 \notin \R$. This technical Lavaurs-type theorem has a number of consequences about the local dynamics of the maps $P$, which we will now state.

\medskip

\subsection{Existence of wandering domains and Pisot numbers}

The Fatou set is the largest open set in $\C^2$ on which the family of iterates $(P^n)_{n \in \N}$ is normal.  A Fatou component $\Omega$ is a connected component of the Fatou set, and it is called \emph{wandering} if for every $(k,m) \in \N \times \N^*$, we have $P^{k+m}(\Omega) \cap P^k(\Omega)=\emptyset$. A non-wandering Fatou component is  a pre-periodic Fatou component. The first examples of polynomial maps with wandering Fatou components were introduced in \cite{ABDPR} by Buff, Dujardin, Peters, Raissy and the first author (see also \cite{astorg2019wandering}); other examples were constructed by Berger and Biebler in \cite{berger2020emergence}, by completely different methods,
for Hénon maps and polynomial endomorphisms of $\ptwo$. 
In the opposite direction, Ji gave in \cite{ji2019nonwandering} and \cite{jin2020nonwandering} sufficient conditions to guarantee the absence of wandering domains near an attracting invariant fiber for a skew-product map.

The examples from \cite{ABDPR} are polynomial skew-products of the form 
$$
(z,w)\mapsto \left(p(z), q(w)+\frac{\pi^2}{4}z\right)
$$
with $p(z)=z-z^2+O(z^3)$ and $q(w)=w+w^2+O(w^3)$, and are not tangent to the identity at the origin. One can   simplify the investigation of these maps by passing to a finite branched cover $y^2=z$. This brings these maps to a form that is tangent to the identity, but with degenerate second order differential at the origin. In particular, these maps are not of the form \eqref{map1} considered in the present paper,
which explain the difference in dynamical features.

\begin{defi}
	We define the rank of a Fatou component $\Omega$ as the maximal rank of $\mathrm{d}h_x$, where $x \in \Omega$
	and $h$ ranges over all Fatou limit functions of  $(P^n)_{n\geq 0}$ on $\Omega$.
\end{defi}

Note that for endomorphism of $\C^2$, any wandering domain either has rank 0 (all Fatou limits are constant) or rank 1. So far, the only known examples of wandering domains in $\C^2$ have rank 0 (that is, the examples constructed in \cite{ABDPR}, \cite{astorg2019wandering} and \cite{berger2020emergence}).
 In other words, Theorem \ref{prop:wd} below gives the first examples of rank 1 wandering domains in complex dimension 2.

\begin{thm}\label{prop:wd} Let $P$ be a map of the form \eqref{map1}, and assume that 
	there exists an $(\alpha_0,\beta_0)$-admissible sequence with converging phase sequence.
	Then $P$ has a wandering domain of rank 1. 
\end{thm}

We are therefore led to the question: for which values of $\alpha$ and $\beta$ does such a sequence exist?
Before stating an answer, recall the definition of Pisot numbers:

\begin{defi}
	A real algebraic integer $\alpha>1$ is called a Pisot number if all of its Galois conjugates are in the open unit disk in $\C$ (in particular, integers $\geq 2$ are Pisot numbers).
\end{defi}

The next definition might not be standard terminology, but it will be convenient for our purposes:

\begin{defi}We say that $\alpha>1$ has the Pisot property if there exist a real number $\zeta$ such that $\|\zeta\alpha^k\|\rightarrow 0$, where
	$\|\cdot\|$ denotes the distance to the nearest integer.
\end{defi}

\medskip

We recall here two classical results from number theory that justify the terminology of "\emph{Pisot property}":
\medskip

{\bf (Pisot):} {\it Let $\alpha>1$ be an algebraic number and $\zeta$ be a non-zero
	real number such that $\|\zeta\alpha^k\|\rightarrow 0$. Then, $\alpha$ is a Pisot number and $\zeta$ lies in the field $\mathbb{Q}(\alpha)$. }

\medskip

{\bf (Viiayaraghavan):} {\it There are only countably many pairs $(\zeta, \alpha)$ of real numbers such that $\zeta\neq 0$, $\alpha>1$, and the sequence $(\{\zeta\alpha^k\})_{k\geq 0}$ has only finitely many limit points. Moreover if  $(\zeta, \alpha)$ is such a pair where $\alpha$ is an algebraic number, then $\alpha$ is a Pisot number and $\zeta$ lies in the field $\mathbb{Q}(\alpha)$. Here $\{\cdot\}$ denotes the fractional part of the number. }

\medskip

In particular, an algebraic number has the Pisot property if and only if it is a Pisot number.
Moreover, it is a long-standing conjecture known as the \emph{Pisot-Viiayaraghavan problem} 
that Pisot numbers are the only \emph{real} numbers with the Pisot property.

\begin{defi} We say that a sequence $(\sigma_k)_{k\geq 0}$ converges to a cycle of period $\ell$ if  the subsequence $(\sigma_{k\ell+j})_{k\geq 0}$ converges for every $0\leq j<\ell$.
\end{defi}

We can now state an almost sharp diophantine condition on $\alpha$ and $\beta$ 
for the existence of an $(\alpha,\beta)$-admissible sequence with converging phase:

\begin{thm}\label{main:pisot}
	Let $\alpha>1$ and $\beta \in \R$. Then
	\begin{enumerate}
		\item There exists an $\alpha$-admissible sequence with phase sequence converging to a cycle if and only if $\alpha$ has the Pisot property. Moreover, in that case there exists an $\alpha$-admissible sequence with phase sequence converging to $0$.
		\item 
		\begin{enumerate}
			\item If there exists an $(\alpha,\beta)$-admissible sequence with phase sequence converging to a periodic cycle, then
			$\alpha$ has the Pisot property.
			\item Conversely, if $\alpha$ has the Pisot property and $\beta = \frac{\alpha-1}{\ln \alpha} \frac{k_1}{k_2}$, 
			where $k_1$ and $k_2$ are coprime integers with $k_2 \geq 1$, then there exists an $(\alpha,\beta)$-admissible sequence whose phase sequence converges to a cycle of period $k_2$.
		\end{enumerate}
	\end{enumerate}
\end{thm}

Note that if the Pisot-Viijayaraghavan conjecture is true, then there exists an $\alpha$-admissible sequence with converging phase sequence if and only if $\alpha$ is a Pisot number.

\medskip

It is natural to ask whether the condition of Theorem \ref{prop:wd} is necessary or not.
In the case that there are no $(\alpha,\beta)$-admissible sequences whose phase sequence converge to a periodic cycle, it means that any wandering Fatou component whose orbit remains in $\bcal_p \times \bcal_{q_0}$ would have 
to remain bounded under a sequence of non-autonomous compositions of generalized Lavaurs maps with non-periodic sequences of phases.
Proving rigorously whether such a thing is possible or not is likely to be very difficult, but 
it seems reasonable to expect that for generic values $\alpha$ it is not the case.

\medskip

If we now specialize to the case of degree 2, Theorems \ref{prop:wd} and \ref{main:pisot} imply that for any Pisot number $\alpha_0>1$, the map 
\begin{equation}\label{eq:deg2sp}
	(z,w) \mapsto  \left(z-z^2,w+w^2+ \left( \frac{1}{4}+\frac{\pi^2}{(\ln \alpha_0)^2 }\right)z^2         \right)
\end{equation}
has a wandering domain of rank 1 (see Figure \ref{fig:deg2}). Those are the first completely explicit examples of polynomial maps with wandering domains, as well as the first examples in degree 2 and the first examples of wandering domains with rank 1.

\begin{figure}
	\centering
	\begin{subfigure}{.5\textwidth}
		\centering
		\includegraphics[width=.9\linewidth]{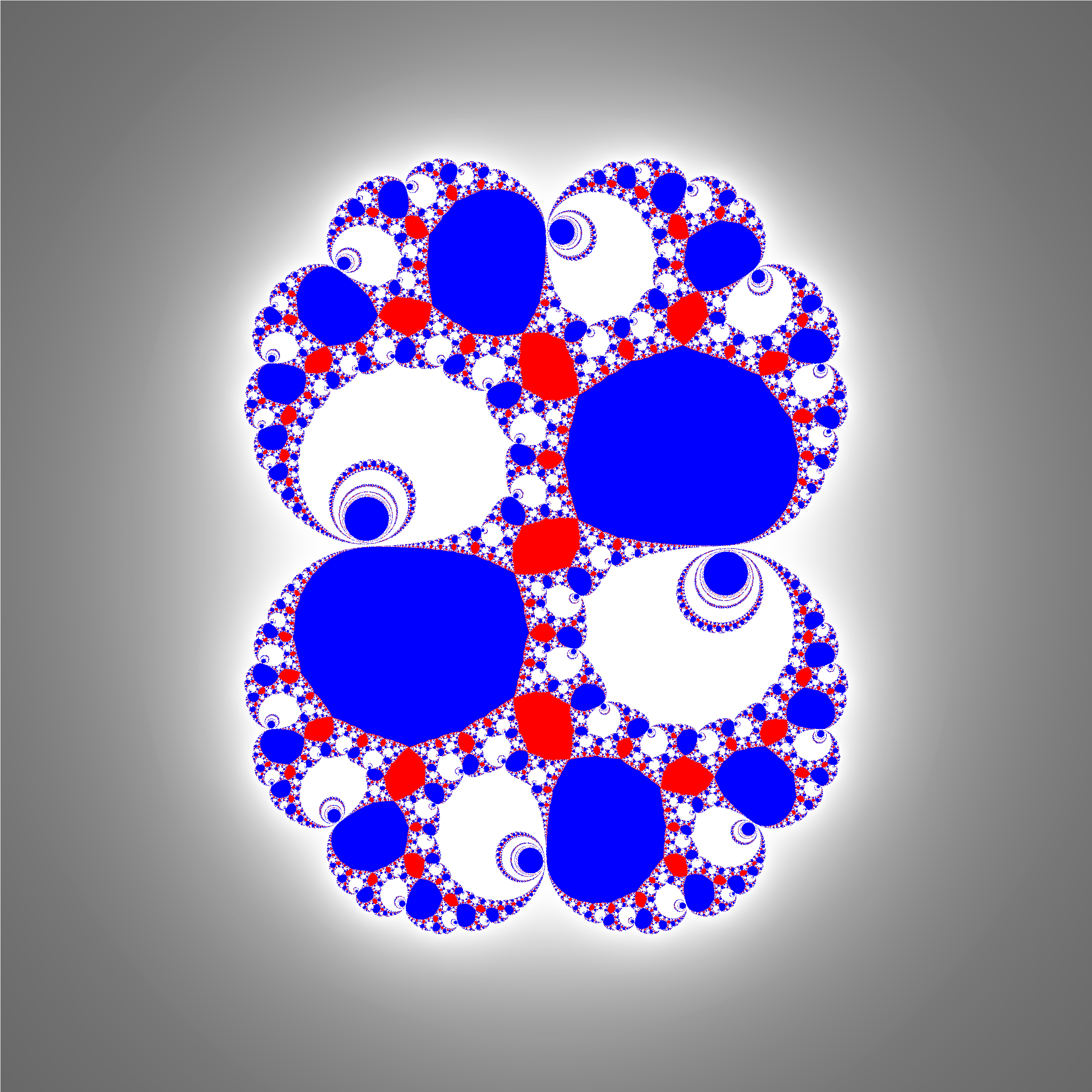}
		\caption{$\alpha_0=2$}
	\end{subfigure}%
	\begin{subfigure}{.5\textwidth}
		\centering
		\includegraphics[width=.9\linewidth]{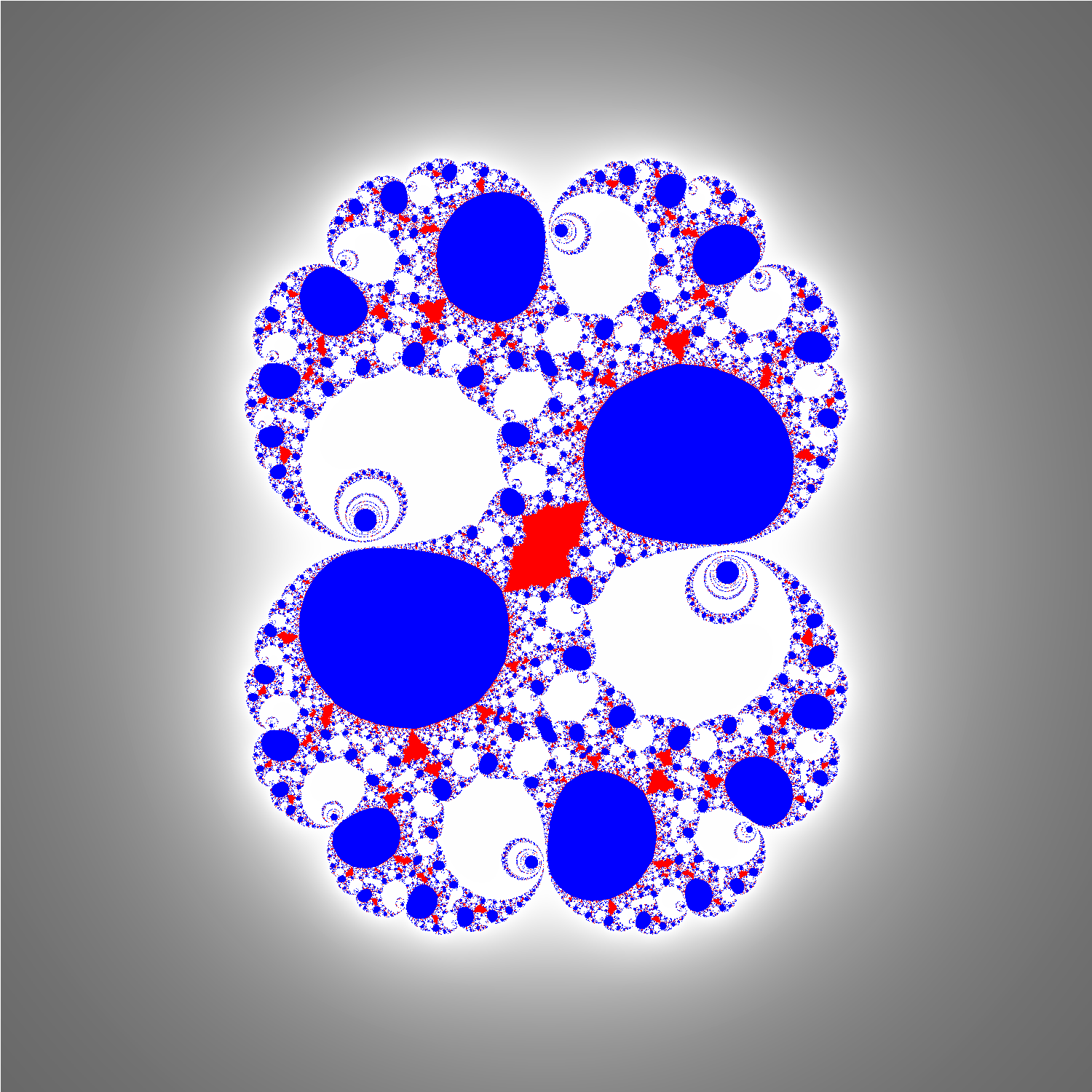}
		\caption{$\alpha_0=\frac{1+\sqrt{5}}{2}$ (non-integer Pisot number)}
	\end{subfigure}
	\caption{Vertical slices $z=\text{constant}$ of quadratic skew-products \eqref{eq:deg2sp} for two different values of $\alpha_0$. In red, wandering domains; in blue, the two parabolic basins; in shades of grey, the basin of infinity.
		Observe that figure (A) is $q_0$-invariant while figure (B) is not.}
	\label{fig:deg2}
\end{figure}

\medskip

Recall that two Fatou components $\Omega_1$ and $\Omega_2$ are in the same grand orbit (of Fatou components) for $P$ if 
there exists $n_1, n_2 \in \N$ such that $P^{n_1}(\Omega_1)=P^{n_2}(\Omega_2)$.
One may ask whether for polynomial endomorphisms of $\ptwo$ there exists a bound on the number of grand orbits of wandering domains that would depend only on the degree. 
The following theorem gives a negative answer:

\begin{thm}\label{prop:many-wd} Let $P$ be of the form \eqref{eq:deg2sp} and let $\alpha_0>1$ be an integer. Then $P$ has countably many distinct grand orbits of rank 1 wandering domains.
\end{thm}

Note that contrary to something like the classical Newhouse phenomenon, we do not use perturbative arguments 
in the proof of Theorem \ref{prop:many-wd}, and the maps considered are completely explicit.
In fact, more precisely, we construct an injective map from the set of hyperbolic components 
in a specific family of modified horn maps into the set of grand orbits of wandering Fatou components of $P$,
see Theorem \ref{thm:manyGO} and the beginning of Section \ref{sec:higherperiod}.

\medskip

\subsection{Topological invariants and horn maps}

We will now investigate a few consequences of the Main Theorem on the 
topological classification of skew-products tangent to the identity.

Recall that in dimension one, the topological classification of germs tangent to the identity is 
just given by the parabolic multiplicity, that is, the order of vanishing of $f-\id$ at the origin.
However, the \emph{analytic} classification of germs tangent to the identity is extremely complicated: 
by a result proved independantly by Écalle and Voronin (\cite{voroninclassification}, \cite{ecalleresurgentes})
the so-called horn maps (also called Écalle-Voronin invariant) are  complete invariants. These horn maps are 
themselves two holomorphic germs fixing $0$ and $\infty$ respectively; see e.g. the Appendix in \cite{ABDPR} for 
more details.

To our knowledge, no complete topological classification is available for germs tangent to the identity
in $\C^2$.
Our results imply that such a classification must  in fact also be complicated
even in the seemingly simple class of skew-products; in fact, it resembles the analytic classification 
for one-dimensional parabolic germs.

\begin{defi}\label{defi:hornP}
	Let us define the lifted horn map of $P$ of phase $\sigma$ by 
	\begin{equation}\label{eq:hornP}
	\tilde \horn_\sigma(Z,W) := (Z, \alpha_0 \cdot \phi_{q_0}^\iota \circ (\phi^o_{q_0})^{-1}(W)+(1-\alpha_0)Z + \sigma)=(Z, \tilde \horn_{Z,\sigma}(W))
	\end{equation}
	The map $\tilde\horn_\sigma$ satisfies the functional relation $\tilde \horn_\sigma(Z+1,W+1)=\tilde \horn_{\sigma}(Z,W)+(1,1)$,
	so it descends to a map $\horn_\sigma$ defined on $\C^2/\Z (1,1)$, 
	which we call  the \emph{horn map} of phase $\sigma$ of $P$.
\end{defi}

\begin{rem}
	Observe that we have the following semi-conjugation: 
	\begin{equation}\label{eq:hornlav}
	(\phi^o_{q_0})^{-1} \circ (\tilde\horn_{Z,\sigma}(W)) = \lcal(\alpha_0,\sigma; z,(\phi^o)^{-1}(W) )
	\end{equation}
	where $Z=\phi_p^\iota(z)$.
\end{rem}

The following is the main result of this subsection:

\begin{thm}\label{thm:top}
	Assume that two maps $P_1$ and $P_2$ of the form \eqref{map1} are topologically conjugated on a neighborhood of the origin, and let $\horn_{\sigma}^i$ denote their respective horn maps.
	Then there exists $\sigma_1, \sigma_2 \in \C$ such that 
	$\horn_{\sigma_1}^1$ and $\horn_{\sigma_2}^2$ are also topologically conjugated on $\C^2/\Z$.
\end{thm}

The following Proposition will be needed
in order to prove Theorem \ref{thm:top}, but it also has an intrinsic interest: 

\begin{prop}\label{coro:mn1=mn2}
	The real numbers $\alpha_0$ and $\beta_0$ from \eqref{eq:c} are topological invariants (and therefore so is $b$).
\end{prop}

Finally, using Theorem \ref{thm:top}, we can obtain:

\begin{coro}	
	If $P_1$ and $P_2$ are topologically conjugated near $(0,0)$, then the number of critical points 
	of $q_i$ in $B_{q_i}$ is the same. In particular, there is no $k \in \N$ such that the local topological conjugacy class of maps of the form \eqref{map1} depend only the $k$-jet of $P$ at the origin.
\end{coro}

In  \cite{abate2005holomorphic}  Abate raised the question whether maps of the form 
$$
(3_{\alpha,\beta, 1}):\quad  f(z,w)= (z+\alpha z^2+(1-\alpha)zw,w+\beta w^2+(1-\beta)zw ), \text{ with } \alpha+\beta\neq 1 \text{ and } \alpha,\beta\neq0
$$
are  topologically conjugated to each other. Using Proposition \ref{coro:mn1=mn2} we can now answer this question negatively. Indeed, observe that for $\alpha=1$ and $\beta\neq 0$ this map is conjugate, via a linear automorphism, to a map 
$$
(z,w)\mapsto \left(z- z^2,w+ w^2+ \frac{1-\beta^2}{4}z^2\right)
$$
which is of the form \eqref{map1}.

\subsection{Fatou components with historic behaviour}

In \cite{berger2020emergence}, Berger and Biebler construct wandering Fatou components $\Omega$ for some maps $f$ (which are Hénon maps or endomorphisms of $\ptwo$) that have 
\emph{historic behaviour}, meaning that for any $x \in \Omega$, the sequence of empirical measures
$$e_n(x) := \frac{1}{n}\sum_{k=1}^n \delta_{f^k(x)}$$
does not converge. 

To our knowledge, these are the only known examples so far of Fatou components for endomorphisms of $\pk$ or for Hénon maps with 
historic behaviour. Note that in the case of the wandering Fatou components constructed in \cite{ABDPR} and 
\cite{astorg2019wandering}, the sequences $(e_n)_{n \in \N}$ converge to the dirac mass centered at the parabolic fixed point
at the origin. In dimension 1, it follows easily from the Fatou-Sullivan classification that no Fatou components of a rational map on $\rs$ can have historic behaviour; and for moderately dissipative Hénon maps, it follows from the classification of Lyubich and Peters \cite{lyubich2014classification} that periodic Fatou components cannot have historic behaviour.

Using the Main Theorem of this paper, we give here new, explicit examples of polynomial skew-products that extend to endomorphisms of $\ptwo$ 
which have a Fatou component with historic behaviour:

\begin{coro}\label{coro:hist}
	Let $P(z,w)= (p(z), q(z,w))$ be a polynomial skew-product satisfying the following properties:
	\begin{enumerate}
		\item $p(z)=z-z^2+O(z^3)$
		\item $P$ has two different fixed points tangent to the identity of the form $(0,w_1)$ and 
		$(0,w_2)$, which both satisfy the conditions that $\alpha_i \in \N^*$ and $\beta_i =0$,
		with the same notations as in the Main Theorem and in appropriate local coordinates.
	\end{enumerate}
	Then $P$ has a 
	Fatou component $\Omega$ with historic behaviour. More precisely, for any $(z,w) \in \Omega$, the sequences $(e_n(z,w))_{n \in \N}$ accumulates on
	$$\mu_1:=\frac{\alpha_1 \alpha_2-\alpha_2}{\alpha_1 \alpha_2-1} \delta_{(0,w_1)}
	+ \frac{\alpha_2-1}{\alpha_1 \alpha_2-1} \delta_{(0,w_2)}$$  
	and on 
	$$\mu_2:= \frac{\alpha_1-1}{\alpha_1 \alpha_2-1} \delta_{(0,w_1)}  +
	\frac{\alpha_1 \alpha_2-\alpha_1}{\alpha_1 \alpha_2-1} \delta_{(0,w_2)}.$$
\end{coro}

More explicitly, these conditions are given by:
\begin{enumerate}
	\item $p(z)=z-z^2+O(z^3)$
	\item $P$ has two different fixed points tangent to the identity of the form $(0,w_1)$ and 
	$(0,w_2)$, with $q_0''(w_i)=2$
	\item $\frac{d^3}{dz^3}_{|z=0}p(z) = \frac{d^3}{d w^3}_{|w=w_i} q_0(w)$
	\item If $b_i:=\frac{1}{2}\frac{\partial^2}{\partial z^2}_{(z,w)=(0,w_i)} q(z,w)$, then $b_i>\frac{1}{4}$, and $\alpha_i:=e^{\frac{2\pi}{\sqrt{4b_i-1}}} \in \N^*$.
\end{enumerate}

\begin{ex}
	With 
	$$p(z):=z-z^2+z^3+z^7$$
	and $q(z,w):=q_0(w)+a(z)$ with
	$$q_0:=w+w^2+w^3-\frac{35}{3}w^4+\frac{39}{2}w^5-13w^6+\frac{19}{6}w^7$$
	and 
	$$a(z):=\left(\frac{1}{4}+\frac{\pi^2}{(\ln 2)^2}\right) z^2(1-z)^2,$$
	 the map $P$ satisfies the conditions above, with $w_1=0$ and $w_2=1$, $\alpha_i=2$ and 
	 $\beta_i=0$.
\end{ex}

Although we believe that the Fatou component constructed in Corollary \ref{coro:hist} is wandering, we were not able to prove so. Note however that if it is not the case,
then this  would be the first example of an invariant (for some iterate of $P$) non-recurrent Fatou component whose limit sets depend on the limit map, which would give an affirmative answer to \cite[Question 30]{lyubich2014classification} for the case $X=\C^2$ and $X=\mathbb{P}^2$.

\subsection*{Structure of the paper}
In Section \ref{sec:parbas}, we recall classical properties of parabolic curves and 
prove Theorem \ref{th:parbas}. In Section \ref{sec:erf}, we introduce some notations, recall some basic facts concerning Fatou coordinates, and introduce approximate Fatou coordinates. We also prove some important estimates on the  error 
function $A$ which measures how close the dynamics is to a translation in these approximate Fatou coordinates.
The Main Theorem is proved in Section \ref{sec:proofmain}. Finally, Sections \ref{proofwd}, \ref{sec:higherperiod}, \ref{section:pisot}, \ref{sec:topinv} and \ref{sec:proofhist}
are devoted to the proofs of Corollary \ref{prop:wd}, Theorem \ref{prop:many-wd}, Theorem \ref{main:pisot}, Theorem \ref{thm:top} and Corollary \ref{coro:hist} respectively.

\subsection*{Acknowledgements} We thank Marco Mancini for invaluable help in writing the code 
used to produce Figure \ref{fig:parhornsquared}, and Arnaud Chéritat for helpful discussions.

\section{Parabolic domains}\label{sec:parbas}

Let $P$ be a holomorphic germ fixing the origin which is tangent to the identity of order $k \ge 2$, i.e. a map with a homogeneous  expansion $P = \mathrm{Id} + P_k + P_{k+1} + \ldots$ where $P_k\not\equiv 0$. We say that $v \in \mathbb{C}^2$ is a \emph{characteristic direction} for $P$ if there exists a $\lambda \in \mathbb{C}$ so that
$P_k(v) = \lambda v$. 
If $\lambda \neq 0$ then $v$ is said to be \emph{non-degenerate} otherwise, it is degenerate.  The \emph{director} of a characteristic direction $v$ is an eigenvalue of a linear operator 
$$
d(P_k)_{[v]}-\id:T_{[v]}\mathbb{P}^1\rightarrow T_{[v]}\mathbb{P}^1.
$$

A \emph{parabolic curve} for $P$ is an injective holomorphic map $\varphi: \Delta\rightarrow\mathbb{C}^2$, satisfying the following properties:
\begin{enumerate}
\item $\Delta$ is simply connected domain in $\C$ with $0\in\partial\Delta$
	\item $\varphi$ is continuous at the origin and $\varphi(0)=(0,0)$,
	\item $\varphi(\Delta)$ is invariant under $P$ and $P^n|_{\varphi(\Delta)}\rightarrow (0,0)$ uniformly on compact subsets.
\end{enumerate}
We say that a parabolic curve is tangent to $[v]\in \mathbb{P}^{1}$  if $[\varphi(\xi)]\rightarrow[v]$ as $\xi\rightarrow 0$ in $\Delta$. This implies that for any point given point $z$ in the parabolic curve the orbit $(P^n(z))$  converges to the origin \emph{tangentially} to $v$, i.e.  $[P^n(z)] \rightarrow [v]$ in $\mathbb{P}^1$. We now recall the following classical result due to Hakim \cite{hakim1994attracting, hakim1998analytic}:

\begin{thm}\label{thm:Hakim}  Let $P: \mathbb{C}^2 \rightarrow \mathbb{C}^2$ be a holomorphic germ fixing the origin which is tangent to the identity of order $k \ge 2$. Then for any non-degenerate characteristic direction $v$ there exist (at least) $k-1$ parabolic curves for $P$ tangent to $[v]$. Moreover if the real part of the director of a non-degenerate characteristic direction $v$ is strictly positive, then  there exists an invariant parabolic domain in which every point is attracted to the origin along a trajectory tangent to $v$.
\end{thm}

From now on let $P$ be a map of the form \eqref{map1} and observe that its characteristic directions are given by the equations
$$
\left\{
\begin{array}{ll}
-z^2&=\lambda z \\
w^2 + bz^2&=\lambda w
\end{array}
\right.
$$

Then, aside from the trivial parabolic curve  $z=0$ with  non-degenerate characteristic direction $(0, 1)$ , there are two parabolic curves $z \mapsto (z,\zeta^\pm(z))$
which are tangent to the non-degenerate characteristic directions $(1, c^\pm)$, where $c^\pm$ 
are the roots of
\begin{equation}\label{equ}
u^2 +u+b=0.
\end{equation}

We break Theorem \ref{th:parbas} into the following two propositions.

\begin{prop}\label{prop:b}
	If $b\in\mathbb{C}- (\frac{1}{4},\infty)$, then the map $P$ has an invariant parabolic domain, in which each point is attracted to the origin along trajectories tangent to one of its non-degenerate characteristic directions.
\end{prop}

\begin{proof} We have two cases:
\medskip

{\bf Case 1: } Let $b\in\mathbb{C}- [\frac{1}{4},\infty)$. A straightforward computation shows that directors of $P$ in the directions
$(0,1)$ and $(1, c^\pm)$ are $-\frac{1}{2}$ and $-1-2c^\pm$ respectively. Since $c^\pm$ are the solutions of the equation \eqref{equ} it follows that $c^\pm=-\frac{1}{2}\pm\frac{x}{2}$, where $x$ is the solution of $x^2=1-4b$, hence $\re(-1-2c^\pm)=\mp \re(x)$. Observe that for $b\in\mathbb{C}- [\frac{1}{4},\infty)$ we have $\re(x)\neq 0$, hence exactly one of the directions $(1, c^\pm)$ has a director with a  strictly positive real part. By Theorem \ref{thm:Hakim} we know that if the real part of the director of a non-degenerate characteristic direction $v$ is strictly positive, then there is an invariant parabolic domain in which each point is attracted to the origin along trajectories tangent to $v$
\medskip

{\bf Case 2:} Let $b=\frac{1}{4}$. First, observe that as $b\rightarrow \frac{1}{4}$, the characteristic directions $(1, c^\pm)$ are getting closer to each other, and in the limit they merge to a single characteristic direction $v=(1,-\frac{1}{2})$. In the terminology of Abate-Tovena \cite{abate2011poincare}, $v$ is an irregular characteristic direction, hence by the result of Vivas \cite[Theorem 1.1]{vivas2012}  there exists an  invariant  parabolic domain, in which each point is attracted to the origin along trajectories tangent to $v$.

\end{proof}

\begin{prop}\label{prop:b>1/4}
	If $b>\frac{1}{4}$, then the map $P$ always has at least two invariant  parabolic domains, where points converge 	non-tangentially to the origin.
\end{prop}

\begin{proof} 
 Let $z\mapsto (z,\zeta^\pm(z))$ be a parabolic curve tangent to a non-degenerate characteristic direction $(1,c^\pm)$. Since it is invariant under $P$, it has to satisfy the equality $q_z(\zeta^\pm(z))=\zeta^\pm(p(z))$. A direct computation then gives us $\zeta^\pm(z):=c^{\pm}z+O(z^2)$. For  $z$ close to the origin, we can define a change of coordinates $\psi^{\pm}(z)=(z,w+\zeta^{\pm}(z))$ which conjugates  our map $P$ to a map of the form
\begin{equation}\label{conjug}
(z,w)\mapsto  (z-z^2+O(z^3),w+w^2+2c^\pm zw +O(zw^2,z^2w,w^3)),
\end{equation}
where $c^\pm=-\frac{1}{2}\pm i\frac{x}{2}$ is the solution of the equation $u^2+u+b=0$ and $x>0$. Note that $(1,0)$ is now a non-degenerate characteristic direction of this map. For the rest of the proof let us focus only on the case of $c^+$; in the case of $c^-$, one can follow computations verbatim with an appropriate change of sign.
\medskip

By making a blow-up $w=uz$ of the map \eqref{conjug}, we obtain  
\begin{equation}\label{blowup}
\tilde{P}(z,u)=(z-z^2+O(z^3),u(1+ ix z)+zu^2+ O(z^2u)).
\end{equation}
where $O(z^2u)$ is holomorphic on some neighbourhood of the origin. Let us define $\mathbb{D}(r,r)=\{z\in\C: |z-r|<r\}$  and $D_r:=\{(z,u)\mid |u|<r,  z\in \mathbb{D}(r,r)\}$ and assume that $r$  is sufficiently small so that $p(\mathbb{D}(r,r))\subset \mathbb{D}(r,r)$.

\begin{lem}\label{lem:blowup}There exists a sequence of real numbers $0<r_j<r$ such that for any 
$(z_0,u_0)\in \mathcal{D}:=\bigcup_{j\geq 1} \{(z,u)\mid |u|<r_j,  z\in \mathbb{D}(r,r\frac{j}{j+1})\}$ we have 
$\tilde{P}^n(z_0,u_0)\in D_r$ for all $n\geq 0$. Moreover, the sequence $\tilde{P}^n(z_0,u_0)$ is bounded away from the origin.
\end{lem}

\begin{proof}[Proof of Lemma \ref{lem:blowup}]
 First observe that for sufficiently small $r>0$, there exists a holomorphic function $h(z)$ such that 
$$\tilde{P}(z,u)=(p(z),\tilde q(z,u))=(z-z^2+O(z^3),ue^{ix z+z^2h(z)}+zu^2+ O(z^2u^2)).
$$

Let $j \in \N^*$ and $z_0 \in K_j:=\overline{\mathbb{D}}(r,r \frac{j}{j+1})$. Note that   we have $ \text{Re}(p^n(z_0))=\frac{1}{n}+O\left(\frac{\ln n}{n^2} \right)$ and $\text{Im}(p^n(z_0))=O\left(\frac{\ln n}{n^2} \right)$, with uniform bounds depending only on $K_j$ for all $n\geq 1$ (see Section \ref{sec:erf}).
Using this, we define
$$
 f_{n}(u):=\text{proj}_2(P(p^{n-1}(z_0),u))=ue^{\frac{ix}{n}+\Theta_n(z_0)}+\frac{u^2}{n}+ O\left(\frac{u^2\ln n}{n^2}\right)
$$
where  $\Theta_n =O\left(\frac{\ln n}{n^2}\right)$ depends only on $z_0$ and is uniformly bounded on $K_j$ and the constant in $O\left(\frac{u^2\ln n}{n^2}\right)$ is uniform on $K_j\times\overline{\mathbb{D}}(0,r)$.
\medskip

We need to prove that there exists an $0<r_j<r$ such that for every $u_0 \in \D(0,r_j)$ and every $z_0 \in K_j$ we have $(z_{n},u_{n}):= \tilde{P}^n(z_0,u_0)\in D_r$ for all $n\geq 1$. In particular we need to prove that $|u_{n}|<r$ for all $n\geq 1$. 
\medskip

Observe that $u_{n}=(f_{n}\circ\ldots \circ f_1)(u_0)$  for all $n\geq1$ and let $U:=\tau(u)=-\frac{1}{u}$. For $n\geq 1$ we define
$$
g_n(U):=(\tau\circ f_n\circ\tau^{-1})(U)=Ue^{-\frac{ix}{n}-\Theta_n}+\frac{1}{n}+ O\left(\frac{\ln n}{n^2},\frac{\ln n}{U n^2}\right).
   $$ 

It suffices to prove that there exists $0<r_j<r$ such that for all $(z_0,U_0)$ where $z_0\in K_j$ and $|U_0|>\frac{1}{r_j}$ we have $|g_n\circ\ldots\circ g_1(U_0)|>\frac{1}{r}$ for all $n\geq 1$.    
      \medskip

Observe that since $x$ is real, there exists $\tilde{C}_j>0$ such that  
$$
\tilde{C}_j^{-1}<\left|e^{-\sum_{k=1}^{n-1} \frac{ix}{k}+\Theta_k}\right|<\tilde{C}_j
$$ 
on $K_j$ for all $n\geq 1$.  By making a non-autonomous change of coordinates 
 $$
 \psi_n(U)=e^{-\sum_{k=1}^{n-1} \frac{ix}{k}+\Theta_k}U,
 $$
 we obtain 
 
 \begin{align*}
 G_n(U)&=\psi_{n+1}^{-1}\circ g_n\circ\psi_n (U)\\
 &=U+\frac{1}{n}e^{\sum_{k=1}^{n} \frac{ix}{k}+\Theta_k}+  O\left(\frac{\ln n}{n^2},\frac{\ln n}{Un^2}\right)\\
 &=U+\frac{1}{n}e^{ix\ln n + ix\gamma+\mathfrak{h}(z_0)}+  O\left(\frac{\ln n}{n^2},\frac{\ln n}{U n^2}\right)
  \end{align*}
 where $\mathfrak{h}:=\sum_{k=1}^{\infty} \Theta_k$ is a holomorphic function of $z_0$. Here, we have used the fact that $\sum_{k=1}^{n} \frac{1}{k}=\gamma+\ln n +O(\frac{1}{n})$ and that $\sum_{k=1}^{n} \Theta_k(z_1)=\mathfrak{h}(z_0)+ O\left(\frac{\ln{n}}{n}\right)$, where the bounds are uniform on $K_j$.

Since $x\neq0$ is real, it follows from Abel's summation formula  that there exists a constant $C>0$ such that 
$$
\left|\sum_{k=1}^n\frac{1}{k}e^{ix\ln k}\right|=\left|\sum_{k=1}^n k^{-(1-ix)}\right|<C
$$
for all $n\geq 1$. This implies that  $G_n\circ\ldots\circ G_1(U)=U+O(1)$ for all $n\geq1$, where the constant in $O(1)$ depends only on $K_j$.

Next observe that $g_n\circ\ldots\circ g_1(U)=\psi_{n+1}\circ G_n\circ\ldots\circ G_1(U)$,
 hence there exists $A_j>0$ such that for all $|U_0|>\frac{1}{r}$ and all $z_0\in K_j$ we have
$$\tilde{C}_j^{-1}|U_0|-A_j<|g_n\circ\ldots\circ g_1(U_0)|<\tilde{C}_j|U_0|+A_j$$
for all $n\geq 1$.

From here it immediately follows that there exists an $0<r_j<r$ such that for every $|U_0|>\frac{1}{r_j}$ we have $|g_n\circ\ldots\circ g_1(U_0)|>\frac{1}{r}$ for all $n\geq 1$. Moreover for every $|U_0|>\frac{1}{r_j}$ the sequence $g_n\circ\ldots\circ g_1(U_0)$ is bounded away from  infinity.

 Therefore we have proven that for any $(z_0,u_0)\in K_j\times \mathbb{D}(0,r_j)$, we have $(z_n,u_n)\in D_r$ for all $n\geq0$, where the sequence $(u_n)_{n \geq 0}$ is bounded away from the origin. This concludes the proof of Lemma \ref{lem:blowup}.
\end{proof}

Let us resume with the proof of Proposition \ref{prop:b>1/4}. 
Let $\Omega:=\{(z,zu)\mid (z,u)\in \mathcal{D}\}$: it is a connected open set whose boundary contains the origin and such that $P(\Omega)\cap \Omega\neq\emptyset$.
From the Lemma above, it immediately follows that  the iterates $P_{|\Omega}^n$ converge to the origin locally uniformly on $\Omega$, which is therefore contained in some invariant parabolic domain. It remains to prove that orbits of points converge non-tangentially to the origin in that parabolic domain. Indeed, let $(z_0,w_0)\in \Omega$ and $(z_{n},w_{n})=P^{n}(z_0,w_0)$ and observe that since $z_n\neq0$, for all $n \in \N$ we have $[z_n:w_n]=[1:\frac{w_n}{z_n }]=[1:u_n]$. From the proof of Lemma \ref{lem:blowup} we can see that every limit map of the iterates $(\tilde{P}^n)$ on $\mathcal{D}$ is of the form $(z,u)\mapsto (0, \eta(z,u))$, where $\eta$ is a non-constant holomorphic function and $\frac{\partial \eta}{\partial u}\not\equiv 0$. Therefore, there is no vector $v\in\mathbb{C}^2$ such that the sequence $[P^n(z,w)]$ would converge to $[v]$ in $\mathbb{P}^1$ for all $(z,w)\in \Omega$.

\end{proof}

\section{Fatou coordinates and properties of the error function}\label{sec:erf}
\subsection{Fatou coordinates}\label{subsec:fatouc}

Consider a holomorphic function $f(z) = z + a_2z^2 + a_3 z^3 + O(z^4)$ where $a_2\neq 0$. For $r>0$ small enough we define incoming and outgoing petals
$$
\mathcal{P}_f^\iota = \{|a_2z+r| < r\} \; \; \mathrm{and} \; \; \mathcal{P}_f^o = \{|a_2 z-r| < r\}.
$$
The incoming petal $\mathcal{P}_f^\iota$ is forward invariant, and all orbits in $\mathcal{P}_f^\iota$ converge to $0$. Moreover, any orbit which converges to $0$ but never lands at $0$ must eventually be contained in $\mathcal{P}_f^\iota$. Therefore we can define the parabolic basin as
$$
\mathcal{B}_f = \bigcup f^{-n} (\mathcal{P}_f^\iota).
$$
The outgoing petal $\mathcal{P}_f^o$ is backwards invariant, with backwards orbits converging to $0$.

On $\mathcal{P}_f^\iota$ and $\mathcal{P}_f^o$ one can define incoming and outgoing Fatou coordinates $\phi_f^\iota: \mathcal{P}_f^\iota \rightarrow \mathbb C$ and $\phi_f^o: \mathcal{P}_f^o \rightarrow \mathbb C$, solving the functional equations
$$
\phi_f^\iota \circ f(z) = \phi_f^\iota(z) + 1 \; \; \mathrm{and} \; \;  \phi_f^o \circ f(z) = \phi_f^o(z) + 1,
$$
where $\phi_f^\iota (\mathcal{P}_f^\iota)$ contains a right half plane and $\phi_f^o (\mathcal{P}_f^o)$ contains a left half plane. By the first functional equation the incoming Fatou coordinates can be uniquely extended to the attracting basin $\mathcal{B}_f$. On the other hand, the inverse of $\phi_f^o$, denoted by $(\phi_f^o)^{-1}$, can be extended to the entire complex plane, still satisfying the functional equation
$$
f \circ (\phi_f^o)^{-1} (Z) = (\phi_f^o)^{-1}(Z+ 1).
$$
This entire function is then called an \emph{outgoing Fatou parametrization}.
 We note that both incoming and outgoing Fatou coordinates are (on the corresponding petals) of the form 
$$\phi_f^{\iota}(z) = -\frac{1}{a_2z} -\mathfrak{b}\log\left(-\frac{1}{a_2z}\right) + o(1)$$
and
$$\phi_f^{o}(z) = -\frac{1}{a_2z}  +\mathfrak{b} \log\left(\frac{1}{a_2z}\right) + o(1)$$
where $\mathfrak{b}:=1-\frac{a_3}{a^2_2}$.   

Finally note that for every $z_0\in\mathcal{B}_f$ we have
 	\begin{align*}
 	z_k:=f^k(z_0)&=(\phi^{\iota}_f)^{-1}(\phi^{\iota}_f(z_0)+k)\\
 	&=-\frac{1}{a_2}\left(k+\mathfrak{b}\ln{k}+\phi^{\iota}_f(z_0)+O\left(\frac{\ln{k}}{k}\right)\right)^{-1}\\
 	&=-\frac{1}{a_2}\left(\frac{1}{k}-\frac{\mathfrak{b}\ln{k}}{k^2}-\frac{\phi^{\iota}_f(z_0)}{k^2}\right)+O\left(\frac{\ln^2{k}}{k^3}\right),
 	\end{align*}
 hence $\text{Re}(a_2z_k)=-\frac{1}{k}+O\left(\frac{\ln{k}}{k^2}\right)$ and $\text{Im}(a_2z_k)=O\left(\frac{\ln k}{k^2} \right)$.

\subsection{The error functions}
Here, we introduce and study properties for one of the main objects to appear in our arguments: the functions $\tilde A(z,w)$, $ A(z,w)$ and $A_0(z)$.

Let $P$ be a skew-product of the form \eqref{map1}, and recall that $v=(1, c^{\pm})$ are two non-degenerate characteristic directions of $P$, where $c^\pm:=-\frac{1}{2}\pm ic$. From Hakim's explicit construction  \cite{hakim1994attracting}, we know that there are two parabolic curves  $z\mapsto (z, \zeta^{\pm}(z)) $ associated to these directions, which are both graphs over a small petal $\mathcal{P}^\iota_p$.  Since parabolic curves are invariant under $P$, it follows that the functions  $\zeta^{\pm}(z)$ satisfy the following functional equation:  
$$
q_z(\zeta^{\pm}(z))=\zeta^{\pm}(p(z)).
$$
From here  we can easily compute the first few terms of their (formal) power series expansion:
$$
 \zeta^{\pm}(z):=c^{\pm}z+\left(c^{\pm} \Theta+\frac{a_3+(b-1)b_{0,3}}{2}\right)z^2+O(z^3),
 $$ 
 where $\Theta:=b_{0,3} +\frac{a_3-b_{0,3}+b_{3,0}}{2b}$

\begin{defi}
	Let $$
	\psi_z (w):=\frac{1}{2ic}\log \left(\frac{\zeta^+(z)-w}{w-\zeta^-(z)} \right) 
	$$
	where $\log$ is the principal branch of logarithm and let 
	$$
	\psi_z^{\iota/o} (w):=\psi_z (w) \pm \frac{\pi}{2c}.
	$$
\end{defi}

Note that with this choice of branch, $\psi_z$ is defined on $\C \backslash L_z$,
where $L_z$ is the real line through  $\zeta^+(z)$ and $\zeta^-(z)$ minus the segment
$[\zeta^-(z), \zeta^+(z)]$. In particular, $\psi_z^\iota$ and $\psi_z^o$ are both
defined in a disk centered at $w=\frac{1}{2}(\zeta^+(z)+\zeta^-(z))$ whose radius is of order $z$.

\begin{defi}
	Let
	\begin{enumerate}
		\item $A(z,w):=\psi_{ p(z)}^{\iota/o} \circ q_{z}(w)-\psi_z^{\iota/o}(w)-z$
		\item $A_0(w):=-\frac{1}{q_0(w)}+\frac{1}{w}-1$
	\end{enumerate}
\end{defi}
Note that the formula for $A(z,w)$ does not depend on whether the ingoing or outgoing coordinate $\psi_z$ is used, and is therefore well defined.

\begin{prop}\label{prop:phianalytic}
	We have that:
	\begin{enumerate}
		\item $A_0(w)=(b_{0,3}-1)w+O(w^2)$ is analytic near zero.
		\item There exists $r>0$ such that for all $z \neq 0$ in a neighborhood of zero, $A(z, \cdot)$ is analytic on the disk		$\D(0,r)$.
	\end{enumerate}
\end{prop}

\begin{proof} The item (1) is an easy computation. For (2), observe that
\begin{align*}
A(z,w)&=\frac{1}{2ic}\log \left(\frac{q_z(w)-\zeta^+(p(z))}{q_z(w)-\zeta^-(p(z))} \right) -\frac{1}{2ic}\log \left(\frac{w-\zeta^+(z)}{w-\zeta^-(z)} \right) -z\\
&=\frac{1}{2ic}\log \left(\frac{q_z(w)-\zeta^+(p(z))}{w-\zeta^+(z)} : \frac{q_z(w)-\zeta^-(p(z))}{w-\zeta^-(z)} \right)-z\\
&=\frac{1}{2ic}\log \left(\frac{q_z(w)-q_z(\zeta^+(z))}{w-\zeta^+(z)} : \frac{q_z(w)-q_z(\zeta^-(z))}{w-\zeta^-(z)} \right)-z.
\end{align*}
It follows that $A(z,w)$ has removable singularities at $w=\zeta^\pm(z)$ unless these are critical points. But up to taking $r>0$
small enough, $\D(0,r)$ contains no critical point of $q_0$.

\end{proof}

\begin{prop}\label{prop:estimateA}
	We have
	$$
	A(z,w)= z A_0(w)  +\left(\Theta+\frac{1}{2}-b_{0,3}\right) z^2+	O(z^3, z^2 w)
	$$
	where the constants in the $O$ are uniform for $(z,w) \in \C^2$ near $(0,0)$
	(with $z \in \mathcal{P}_p^\iota$).
\end{prop}

\begin{proof} Let $w\in K$ be a compact in $\mathbb{C}^*$. By a straightforward computation we obtain 
\begin{align*}
 \frac{1}{2ic}\log\left(\frac{w-\zeta^+(z)}{w-\zeta^-(z)} \right)&=\frac{1}{2ic}\left( \frac{\zeta^-(z)-\zeta^+(z)}{w}-\frac{(\zeta^+(z))^2-(\zeta^-(z))^2}{2w^2}\right)+O(z^3)\\
 &=-\frac{z}{w}-\frac{\Theta z^2}{w}+\frac{z^2}{2w^2}+O(z^3).
 \end{align*}
 Using this we can now show that
\begin{align*}
 \psi_{ p(z)}^{\iota/o} \circ q_{z}(w)&=-\frac{ p(z)}{ q_{z}(w)}-\frac{\Theta(p(z))^2}{ q_{z}(w)}+\frac{( p(z))^2}{2( q_{z}(w))^2}+O(z^3)\\
 &=-\frac{ z-z^2}{ q_{0}(w)}- \frac{\Theta z^2}{q_{0}(w)}+\frac{ z^2}{2(q_{0}(w))^2}+O(z^3).
\end{align*}
This implies that
\begin{align*}
A(z,w)&=zA_0(w)+\Theta z^2\left(\frac{1}{w}-\frac{1}{q_{0}(w)}\right)\\
&+\frac{z^2}{2}\left(\frac{1}{(q_{0}(w))^2}-\frac{1}{w^2}+\frac{2}{q_0(w)}\right)+O(z^3)\\
&=zA_0(w)+\Theta z^2+\frac{z^2}{2}(1-2b_{0,3})+O(z^3,z^2w)\\
&= z A_0(w)  +\left(\Theta+\frac{1}{2}- b_{0,3}\right) z^2+	O(z^3, z^2 w).
\end{align*}
Here, we used the fact that $A(z,w)$ is analytic, hence all terms of $w$ with the negative power are cancelled.  

Note that the constant in the $O(z^3, z^2 w)$ a priori depends on $K \subset \C^*$.
	Let $\phi_z(w):=\frac{A(z,w)-z A_0(w)}{z^2}$ and note that by Proposition \ref{prop:phianalytic} it is holomorphic on $\D(0,r)$. We have proved that
	for all compact $K \subset \C^*$, for all $w \in K$, and for all small $z \neq 0$ with $\mathrm{Re}(z)>0$, we have
	$|\phi_z(w)| \leq C_K$. By taking $K=\{ |w|=\frac{r}{2}  \}$ we therefore obtain the same
	estimate $|\phi_z(w)| \leq C_K$ for all $|w| \leq \frac{r}{2}$ because of the maximum modulus principle.
	This gives the desired uniformity.

\end{proof}

\begin{defi}\label{defi:R}
	As in \cite{ABDPR}, let $\nu\in(\frac{1}{2},\frac{2}{3})$ and
	\begin{enumerate}
		\item $r_{z}:=|z|^{1-\nu}$
		\item $\mathcal{R}_z:=\{W \in \C : \frac{r_z}{10}<\re(W)<\frac{\pi}{c}-\frac{r_z}{10} \text{ and } -\frac{1}{2} < \im(W)<\frac{1}{2} \}$
	\end{enumerate}
\end{defi}

\begin{defi}
	Let $\chi_z(W)=W+ ( b_{0,3}-1) R(z,W)$, where 
	\begin{equation}
	R(z,W):=cz e^{W} F_c(W)
	\end{equation}
	and $F_c$ is the primitive on $\rcal_0$ of $W \mapsto e^{-W} \cot(cW)$ vanishing at $\frac{\pi}{2c}$.
\end{defi}

A straightforward computation shows that 	$R(z,W)$ is a solution of the linear PDE
\begin{equation}
-z \frac{\partial R}{\partial z} + \frac{\partial R}{\partial W} =c z \cot(cW).
\end{equation}

\begin{lem}\label{lem:invphi}
	We have
	$$(\psi_{z}^{\iota/o})^{-1}(W)=-c z \cot(cW) - \frac{z}{2}+{O(z^2 \cot(cW),z^2)} $$
\end{lem}

\begin{proof}
	We have:
	\begin{equation}
		(\psi_z^{\iota/o})^{-1}(W)=\frac{\zeta^+(z)-\zeta^-(z) e^{2icW}}{1- e^{2icW}}
	\end{equation}
	and using the fact that $ \zeta^{\pm}(z):=(-\frac{1}{2}\pm ic)z+\left(\frac{a+(b-1) b_{0,3}-\Theta}{2}\pm ic\Theta\right)z^2+O(z^3)$, we find
	\begin{equation}
		(\psi_z^{\iota/o})^{-1}(W)=-cz \cot(cW)-\frac{z}{2}+{O(z^2 \cot(cW),z^2)}.
	\end{equation}
\end{proof}

\begin{lem}\label{lem:dR}
	Assume that $\psi_z^\iota(w) \in \rcal_z$, and let $W:=\psi_z^\iota(w)$, $W_1:=\psi_{p(z)}^\iota \circ q_z(w)$ and $z_1:=p(z)$. Then 
	\begin{equation}
	\left|R(z_1, W_1)-R(z,W) - cz^2 \cot(cW) \right| = O\left(|z|^{2+\delta}\right)
	\end{equation}
	for some $\delta>0$.
\end{lem}

\begin{proof}
	Let $x:=(z,W)$ and $h:=(z_1,W_1) - (z,W)$. Then by Taylor-Lagrange's formula, we have
	\begin{equation}
	R(x+h)-R(x) - dR_x(h) = \int_0^1 \frac{(1-t)^2}{2} d^2 R_{x+th}(h,h) dt
	\end{equation}
	and 
	\begin{align*}
	d^2 R_y(h,h) &= R_{zz}(y) h_1^2 + 2 R_{zW}(y) h_1 h_2 + R_{WW}(y) h_2^2 \\
	&= 0 + O(z^3\cot(cW) ) + O(z^3 \cot^2(cW))
	\end{align*}
	(Here, $R_{zz}:=\frac{\partial^2 R}{\partial z^2}$, etc.).
	Since $W \in \rcal_z$ by assumption, we have $z^3 \cot^2(cW) =O(|z|^{1+2\nu})=O(|z|^{2+\delta})$ for some $\delta>0$.
	Therefore 
	$$|R(z_1, W_1)-R(z,W) - dR_{x}(h) |=  O\left(|z|^{2+\delta}\right).
	$$
	It now remains to compare $dR_x(h)$ and $cz^2 \cot(cW)$.
	First, note that 
	$$h=(-z^2+O(z^3), z+O(zw))= (-z^2 + O(z^3), z+ O(z^2 \cot(cW)).$$
	Therefore
	\begin{align*}
	dR_x(h) &= R_z(x) h_1 + R_W(x) h_2 \\
	&=-z^2  R_z(x) + z R_W(x) + O(z^3 R_W, z^2 \cot(cW) R_W)  \\
	&=cz^2 \cot(cW) + O(z^3 \cot(cW), z^3 \cot^2(cW))
	\end{align*}
	hence we have
	$$
	|dR_x(h)- cz^2 \cot(cW)|=  O\left(|z|^{2+\delta}\right).
	$$
\end{proof}

\begin{defi} We define $\tilde A(z,w) :=\chi_{p(z)} \circ \psi_{p(z)}^\iota \circ q_{z}(w) - \chi_z \circ \psi_z^\iota(w)-z $.
\end{defi}

\begin{prop}[Almost translation property]\label{prop:almosttransv2}
	There exists $\delta>0$ (depending only on the choice of $\nu$) such that
	$$|\tilde A(z,w)-\Lambda z^2|=O\left(|z|^{2+\delta}\right)$$
	for all $(z,w)$ such that $\psi_z^\iota(w) \in \mathcal{R}_z$, where $\Lambda :=\Theta+1-\frac{3 b_{0,3}}{2}$
\end{prop}

\begin{proof} Let  $z_1:=p(z)$,  $W:=\psi_z^\iota(W)$ and $W_1:=\psi_{z_1}^\iota \circ q_z(w)$. 
We  have
	\begin{align*}
	\tilde A(z,w) &=\chi_{z_1} \circ \psi_{z_1}^\iota \circ q_{z}(w) - \chi_z \circ \psi_z^\iota(w)-z \\
	&=\psi_{z_1}^\iota \circ q_{z}(w) -  \psi_z^\iota(w)-z +  ( b_{0,3}-1) (R(z_1, W_1) - R(z,W)).
	\end{align*}
	By Lemma \ref{lem:dR}
	\begin{align*}
	|\tilde A(z,w)-A(z,w) - cz^2 ( b_{0,3}-1) \cot(cW)|= O\left(|z|^{2+\delta}\right)
	\end{align*}

	On the other hand, by Proposition \ref{prop:estimateA} we have
	\begin{align*}
	A(z,w)&=z A_0(w) + \left(\Theta+\frac{1}{2}- b_{0,3}\right)z^2 + O(z^2 w,z^3)\\
	&=( b_{0,3}-1)zw+ \left(\Theta+\frac{1}{2}- b_{0,3}\right)z^2 + O(zw^2,z^2 w,z^3)\\
	\end{align*}
	so  using Lemma \ref{lem:invphi}:

$$
	A(z,w) =(1- b_{0,3}) cz^2 \cot(cW) + z^2 \left(\Theta+1-\frac{3 b_{0,3}}{2} \right)+  O\left(zw^2,z^2 w,z^3,z^3 \cot(cW)\right) 
	$$
	
	Putting all of these estimates together, we get:
	\begin{align*}
	|\tilde A(z,w) - \Lambda z^2| =O\left(|zw^2|,|z^2 w|,|z|^3,|z|^{2+\delta}\right) 
	\end{align*}
	Finally, note that since by assumption $\psi_z(w) \in \mathcal{R}_z$, we have 
	$|w| = O\left(|z|^{\nu}\right)$. Moreover, recall that $\nu>\frac{1}{2}$, so that:
	\begin{itemize}
		\item $|zw^2| = O\left(|z|^{1+2\nu}\right)$
		\item $|z^2w|=O\left(|z|^{2+\nu} \right)$
	\end{itemize}
\end{proof}

\begin{lem}\label{lem:DLFc}
	As $W \to 0$ in $\rcal_0$, we have
	\begin{equation}
		F_c(W)=\frac{1}{c} \log (cW) - \frac{1}{c}\int_0^{\frac{\pi}{2c}} e^{-u} \ln \sin(cu) du + o(1)
	\end{equation}
	Similarly, as $W \to \frac{\pi}{c}$ in $\rcal_0$, we have:
	\begin{equation}
		F_c(W) =e^{-\frac{\pi}{c}}\frac{1}{c} \log\left(\pi - cW\right) +  + \frac{1}{c} \int_{\frac{\pi}{2c}}^{\frac{\pi}{c}} e^{-u} \ln \sin(cu) du + o(1)
	\end{equation}
\end{lem}

\begin{proof}
	Recall that $F_c(W)=\int_{\frac{\pi}{2c}}^{W} e^{-u} \cot(cu) du$. An integration by parts gives:
	\begin{align*}
		F_c(W) &=\frac{1}{c} e^{-W}\log \sin(cW) + \frac{1}{c}\int_{\frac{\pi}{2c}}^W e^{-u} \log \sin(cu) du\\
	\end{align*}
	from which it follows that as $W \to 0$:
	\begin{align*}
		F_c(W)&=\frac{1}{c} \log (cW) + o(1) - \frac{1}{c}\int_0^{\frac{\pi}{2c}} e^{-u} \ln \sin(cu) du \\
	\end{align*}
	and as $W \to \frac{\pi}{c}$:
	\begin{align*}
		F_c(W)&=e^{-\frac{\pi}{c}}\frac{1}{c} \log(\pi - cW) + \frac{1}{c} \int_{\frac{\pi}{2c}}^{\frac{\pi}{c}} e^{-u} \ln \sin(cu) du + o(1)
	\end{align*}
\end{proof}

\section{Proof of the main theorem}\label{sec:proofmain}

We begin this section by explaining how the map $\psi_z$, defined in the previous section, transforms the complex plane.

 Let $D_z$ be the disk of radius $\frac{1}{2}|\zeta^+(z)-\zeta^-(z)|=c|z|+O(z^2)$ centered at  $\frac{1}{2}(\zeta^+(z)+\zeta^-(z))$. Let $\mathcal{S}(z,R)$ be the union of the two disks of radius $R$ that both contain the points $\zeta^+(z),\zeta^-(z)$ on their boundary. The radius $R$ will be a sufficiently small number, to be fixed later. The definition of $\mathcal{S}(z,R)$ of course only makes sense when the distance between $\zeta^+(z)$ and $\zeta^-(z)$ is less than $2R$, which once $R$ is fixed will be satisfied for $z$ sufficiently small. Our choice of $R$ will depend on the map $q_0$, but not on $z$.

The line $L_z$ through  $\zeta^+(z)$ and $\zeta^-(z)$ cuts the complex plane into the left half plane $H^\iota_z$ and the right half plane $H^o_z$. We define $\mathcal{S}^{\iota/ o}(z,R):= \mathcal{S}(z,R) \cap H^{\iota/o}_z$. The map $\psi_z$ maps the disk $D_z$ to the shaded strip $[-\frac{\pi}{4c},\frac{\pi}{4c}]\times i\mathbb{R}$.  The image of $\mathcal{S}(z,R)$ is bounded by two vertical lines, intersecting the real line in points of the form $\mp\frac{\pi}{2c}\pm O(z)$, see Figure \ref{fig:1}.  Next we define $\mathcal{P}^{\iota/ o}_R:=\mathbb{D}(\mp R,R)$ and observe that  $\mathcal{S}^{\iota/o}(z,R) \xrightarrow{z\rightarrow 0} \mathcal{P}^{\iota/ o}_R$. 

\medskip

\begin{figure}[h!]
  \includegraphics[scale=1.1]{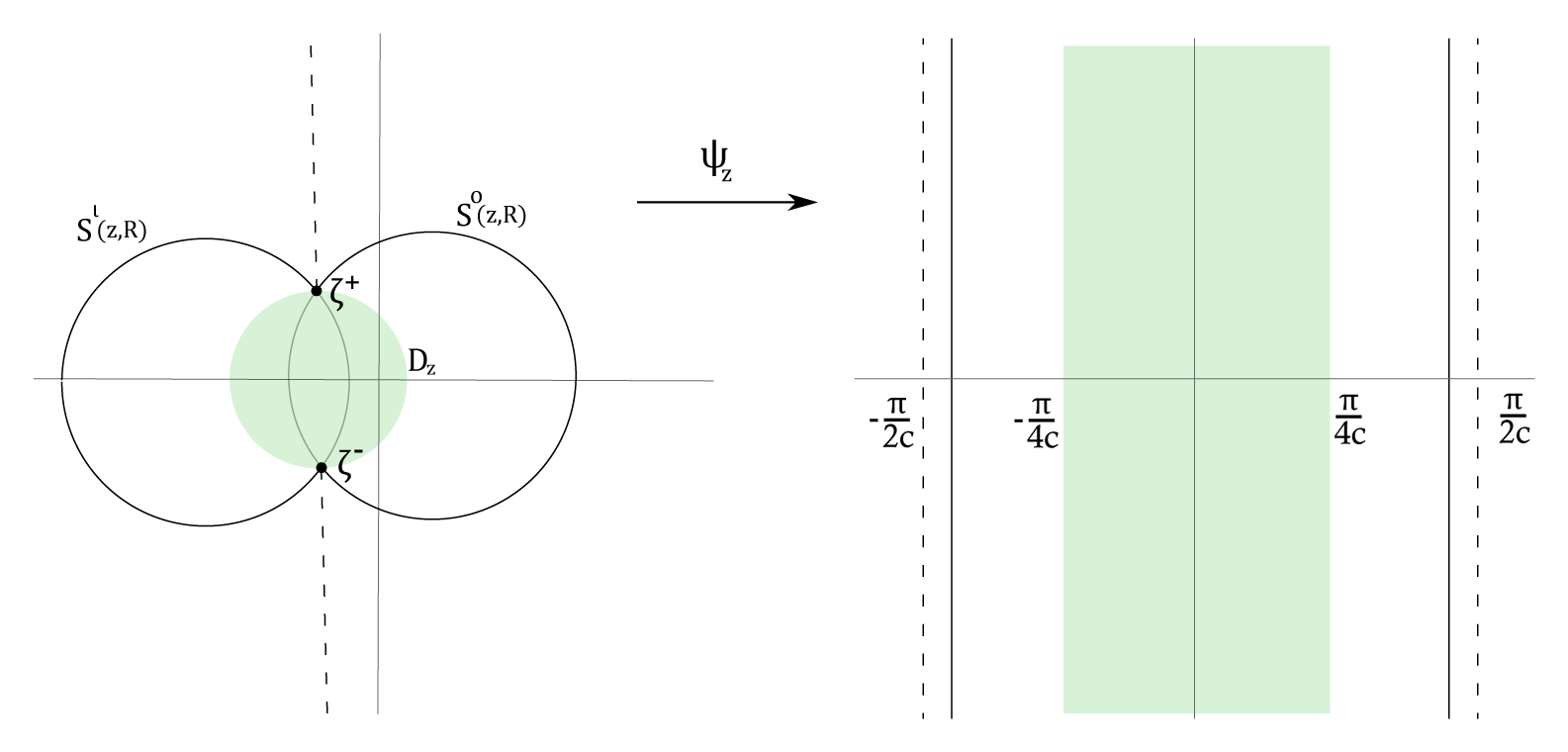}
  \caption{}
  \label{fig:1}
\end{figure}

\medskip 

{\bf Key observation:} There are positive real constants $r_0$, $R$, $s$, $t$, $\delta$ such that:
\begin{enumerate}[label=(\roman*)]
\item The invariant curves $z\mapsto(z,\zeta^{\pm}(z))$ are graphs over the disk $\mathbb{D}(r_0,r_0) \subset \mathcal{B}_p$.
\item We have $$
\left[-\frac{\pi}{2c}+s |z|,\frac{\pi}{2c}-s|z|\right]\times i\mathbb{R}\subset \psi_z(\mathcal{S}(z,R))\subset\left[-\frac{\pi}{2c}+ t|z|,\frac{\pi}{2c}-t|z|\right]\times i\mathbb{R}$$
for all  $z\in \mathbb{D}(r_0,r_0)$. 
\item For every compact $K\subset\mathcal{P}^{\iota}_R $ there exists $0<r'<r_0$ such that 
$$
K\subset\mathcal{S}^{\iota}(z,R)\subset \mathcal{P}^{\iota}_{2R}
$$
 for all $z\in \mathbb{D}(r',r')$,
 
 \item $|\tilde A(z,w)-\Lambda z^2|<|z^{2+\delta}|$ for all $(z, w)\in \mathbb{D}(r_0,r_0)\times\mathbb{D}(0,4R)$ (see Proposition \ref{prop:almosttransv2}).
\item  the inverse $q_0^{-1}(w)$ is well defined on $\mathbb{D}(0,4R)$. 
\item $q_0(\mathcal{P}^{\iota}_{2R})\subset\mathcal{P}^{\iota}_{2R}$ and $q_0^{-1}(\mathcal{P}^{o}_{2R})\subset\mathcal{P}^{o}_{2R}$
\end{enumerate}

\begin{rem} Recall that $\psi_z^\iota=\psi_z+\frac{\pi}{2c}$, therefore $(ii)$ implies $\psi_z^\iota(\mathcal{S}(z,R))\subset \mathcal{R}_z$ for all $z\in \mathbb{D}(r,r)$ assuming that $r>0$ is sufficiently small (recall that $\mathcal{R}_z$ was introduced in Definition \ref{defi:R}). 
\end{rem}

We are now ready to start with the proof:

\medskip
{\bf Fixing a constant:} We now fix these constants  $r_0$, $R$, $s$, $t$, $\delta$ and define $k_n:=[n^{\nu}]$ for the constant $\nu \in (\frac{1}{2},\frac{2}{3})$ already defined in Definition \ref{defi:R}.

\medskip

{\bf Notation:} Given a point $(z_0,w_0)\in\mathcal{B}_p\times \mathcal{B}_{q_0}$ and an integer $n>0$ we will write $\eps_j:=p^{n+j}(z_0)$ and $w_j:=q_{\eps_j} \circ q_{\eps_{j-1}} \circ \ldots \circ q_{\eps_1}(w_0)$. 
\medskip

{\bf Fixing a compact:} For the rest of this section  we fix a compact subset $K'\times K \subset \mathcal{B}_p\times \mathcal{B}_{q_0}$. Let $n_0$ be sufficiently large integer so that $p^{k_n}(K')\subset\mathbb{D}(r_0,r_0)$ for all $n>n_0$. By taking even larger $n_0$ if necessary we may assume that $w_{k_n}\in\mathcal{P}^{\iota}_R$ and therefore by $(iii)$ above $w_{k_n}\in\mathcal{S}^\iota(\eps_{k_n},R)$ for all  $(z_0,w_0)\in K'\times K$ and all $n\geq n_0$. Finally we fix a point $(z_0,w_0)\in K'\times K$. 

\begin{rem}Unless otherwise stated, all the constants appearing in estimates depend only on the compact $K'\times K$, but not on the point $(z_0,w_0)$ nor the integer $n$.
\end{rem}

\subsection{Entering the eggbeater}

\begin{lem}\label{lem:init}
	We have $\phi_{q_0}^\iota(w_{k_n})=\phi_{q_0}^\iota(w_0)+k_n + o(1)$  and hence $w_{k_n}=-\frac{1}{k_n}+O\left(\frac{\ln n}{k_n^2}\right).$
\end{lem}

\begin{proof}
For $0\leq j\leq k_n$ we have
	\begin{align*}
		\phi_{q_0}^\iota(w_{j+1})&=\phi^\iota_{q_0}({q_0}(w_j)+ b\eps_j^2 + O(\eps_j^3)) \\
		&=\phi_{q_0}^\iota(w_j)+1+O((\phi_{q_0}^\iota)'(w_j) \eps_j^2 ) \\
		&=\phi_{q_0}^\iota(w_j)+1 + O\left(\frac{\eps_j^2}{w_j^2} \right) \\
		&=\phi_{q_0}^\iota(w_j)+1+ O\left(\frac{k_n^2}{n^2}\right)
	\end{align*}
	Therefore by induction, $\phi_{q_0}^\iota(w_{k_n})=\phi_{q_0}^\iota(w_0)+k_n + O(\frac{k_n^3}{n^2})$,	and $\frac{k_n^3}{n^2}=o(1)$ by the choice of $k_n$. Final conclusion follows from the fact that  
	$$\phi^{\iota}_{q_0}(w)=-\frac{1}{w}+(1- b_{0,3})\log(-w) + o(1).$$
\end{proof}

\begin{lem}\label{prop:error1}
	We have
	$$
	\psi_{\epsilon_{k_n}}^{\iota}(w_{k_n})=-\frac{\epsilon_{k_n}}{w_{k_n}}+{\frac{\epsilon_{k_n}^2}{2w_{k_n}^2}}+o\left(\epsilon_{k_n}\right).
	$$
\end{lem}
\begin{proof}
This follows directly from the computation in the proof of Proposition \ref{prop:estimateA}.
\end{proof}

\begin{defi}[Approximate Fatou coordinate] Let $\Phi_z:= \chi_{z} \circ \psi_{z}^\iota$.
\end{defi}

\begin{lem}[Comparison with incoming Fatou coordinates]\label{lem:comp}
	We have 
	\begin{equation*}
		\frac{1}{\eps_{k_n}}\Phi_{\eps_{k_n}}(w_{k_n})=\phi^\iota_{{q_0}}(w_{k_n}) +{\frac{k_n^2}{2n}}+ (1- b_{0,3}) \ln n +  E^\iota + o(1)
			\end{equation*}
where $E^\iota:=( b_{0,3}-1)  \left(\ln c - \int_0^{\frac{\pi}{2c}} e^{-u} \ln \sin(cu) du   \right)$.
\end{lem}

\begin{proof}
	Recall that by Lemma \ref{lem:init} and Lemma \ref{prop:error1}	we have
	$$w_{k_n}=-\frac{1}{k_n}+O\left(\frac{\ln n}{k_n^2}\right)$$
	and
	$$\psi_{\eps_{k_n}}^\iota(w_{k_n}) = -\frac{\eps_{k_n}}{w_{k_n}}+\frac{\eps_{k_n}^2}{2w_{k_n}^2}+o(\eps_{k_n}).$$
	
	Next, we have:
\begin{align*}
	\frac{1}{\eps_{k_n}}\Phi_{\eps_{k_n}}(w_{k_n}) &= \frac{1}{\eps_{k_n}}\chi_{\eps_{k_n}} \circ \psi_{\eps_{k_n}}^\iota(w_{k_n}) \\
	&=\frac{W}{\eps_{k_n}} - c  (1- b_{0,3}) e^{W} F_c(W) \text{ , \quad \quad where } W:= \psi_{\eps_{k_n}}^\iota(w_{k_n})\\
	&= -\frac{1}{w_{k_n}}+\frac{k_n^2}{2n} -c  (1- b_{0,3}) e^{W} F_c(W) +o(1) 
\end{align*}	
and by Lemma \ref{lem:DLFc},
\begin{align*}
c e^{W} F_c(W)&= \log\left(-\frac{\eps_{k_n}}{w_{k_n}}\right) +\ln c - \int_0^{\frac{\pi}{2c}} e^{-u} \ln \sin(cu) du   +o(1)\\
&= -\log(-w_{k_n})-\ln{n} +\ln c - \int_0^{\frac{\pi}{2c}} e^{-u} \ln \sin(cu) du   +o(1).
\end{align*}

Putting all together we get
	\begin{align*}
		\frac{1}{\eps_{k_n}} \Phi_{\eps_{k_n}}(w_{k_n}) &= -\frac{1}{w_{k_n}} + (1- b_{0,3}) \log(-w_{k_n})+\frac{k_n^2}{2n} + (1- b_{0,3}) \ln n + E^\iota + o(1) \\
		&=\phi^\iota_{q_0}(w_{k_n})+\frac{k_n^2}{2n} +(1- b_{0,3}) \ln n + E^\iota + o(1).
 	\end{align*}
 
\end{proof}

 \subsection{Passing through the eggbeater}

\begin{defi}\label{defi:mn} Let $\alpha_0,\beta_0$ be as in \eqref{eq:c} and define 
$M_n:= \lfloor(\alpha_0-1)n+ \beta_0\ln n \rfloor$
where  $\lfloor \cdot\rfloor$ is the floor function. Let $\ell_n:= \lfloor e^{\frac{\pi}{c}}k_n\rfloor$ and $\rho_n:=\{(\alpha_0-1)n+ \beta_0\ln n  \}$, where $\{\cdot\}$ denotes the fractional part. Finally we define $W_{j}:=\Phi_{\eps_{j}}(w_{j})$   
	
\end{defi}

\begin{lem}\label{lem:sumeps}	We have 
\begin{align*}
W_{k_n}+\sum_{j=k_n}^{M_n-\ell_n-1} \eps_j + \Lambda \eps_j^2= \frac{\pi}{c}+\frac{G_n}{n}+ o\left(\frac{1}{n}\right)
\end{align*}
where 
$$G_n:=-e^{-\frac{\pi}{c}}\ell_n+{\frac{k_n^2}{2n}}+(1- b_{0,3})e^{-\frac{\pi}{c}}\ln{n}-e^{-\frac{\pi}{c}}\rho_n+
\phi^\iota_{q_0}(w_{0})+\tilde{C}$$
and 
$$\tilde{C}:=(1-a)e^{-\frac{\pi}{c}}\frac{\pi}{c} +(1-e^{-\frac{\pi}{c}})\left(\Theta+\frac{3}{2}(1- b_{0,3})+(a-1)-\phi^{\iota}_p(z_0)\right)+	E^\iota .$$
\end{lem}
\begin{proof}
First recall that by Lemma \ref{lem:init} and Lemma \ref{lem:comp} we have:
$$W_{k_n}=\frac{1}{n}\left(\phi^\iota_{q_0}(w_{0})+k_n+\frac{k_n^2}{2n} +(1- b_{0,3}) \ln n +E^\iota\right) + o\left(\frac{1}{n}\right).$$
		
Next recall that  
$$
\eps_j=\frac{1}{n+j}-\frac{(1-a)\ln(n+j)+\phi^{\iota}_p(z_0)}{(n+j)^2}+O\left(\frac{\ln^2 n}{n^3}\right)
$$

and observe that  by the Euler-MacLaurin formula we get
\begin{align*}
\sum_{j=k_n}^{M_n-\ell_n-1} \eps_j &=\int^{M_n-\ell_n}_{k_n} \eps_j  dj+ \frac{1}{2}(\eps_{k_n}-\eps_{M_n-\ell_n})+o\left(\frac{1}{n}\right)\\
&=\int^{M_n-\ell_n}_{k_n} \eps_j  dj+ \frac{1}{2n}(1-e^{-\pi/c})+o\left(\frac{1}{n}\right)
\end{align*}
and
\begin{align*}
\sum_{j=k_n}^{M_n-\ell_n-1}  \eps_j^2 &= \eps_{k_n}-\eps_{M_n-\ell_n} + o\left(\frac{1}{n}\right)=\frac{1}{n}(1-e^{-\frac{\pi}{c}}) + o\left(\frac{1}{n}\right).
\end{align*}	
Furthermore we have
\begin{align*}
&\int^{M_n-\ell_n}_{k_n} \eps_j  dj=\ln\left(\frac{n+M_n-\ell_n}{n+k_n}\right)+((1-a)+\phi^{\iota}_p(z_0))\left(\frac{1}{n+M_n-\ell_n}-\frac{1}{n+k_n}\right)\\
&\qquad +(1-a)\left(\frac{\ln(n+M_n-\ell_n)}{n+M_n-\ell_n}-\frac{\ln(n+k_n)}{n+k_n}\right)+o\left(\frac{1}{n}\right)\\
&=\frac{\pi}{c}+( b_{0,3}-a)(1-e^{-\frac{\pi}{c}})\frac{\ln n}{n}-e^{-\frac{\pi}{c}}\frac{1}{n}\rho_n-\frac{k_n}{n}-e^{-\frac{\pi}{c}}\frac{\ell_n}{n}+\frac{1}{n}(1-a)\left(e^{-\frac{\pi}{c}}-1\right)\\
&\qquad+\frac{1}{n}\left(e^{-\frac{\pi}{c}}-1\right)\phi^{\iota}_p(z_0)+\frac{1}{n}(1-a)e^{-\frac{\pi}{c}}\frac{\pi}{c}
 +(1-a)(e^{-\frac{\pi}{c}}-1)\frac{\ln{n}}{n}+o\left(\frac{1}{n}\right)\\
&=\frac{\pi}{c}+( b_{0,3}-1)(1-e^{-\frac{\pi}{c}})\frac{\ln n}{n}-\frac{k_n}{n}-e^{-\frac{\pi}{c}}\frac{\ell_n}{n}\\
&\qquad+\frac{1}{n}\left((1-a+\phi^{\iota}_p(z_0))\left(e^{-\frac{\pi}{c}}-1\right)+(1-a)e^{-\frac{\pi}{c}}\frac{\pi}{c}-e^{-\frac{\pi}{c}}\rho_n\right)+o\left(\frac{1}{n}\right).
\end{align*}	

Putting all together we obtain 		

\begin{align*}
W_{k_n}+\sum_{j=k_n}^{M_n-\ell_n-1} \eps_j + \Lambda  \eps_j^2 &=
\frac{\pi}{c}-e^{-\frac{\pi}{c}}\frac{\ell_n}{n}+(1- b_{0,3})e^{-\frac{\pi}{c}}\frac{\ln n}{n}+{\frac{k_n^2}{2n^2}}\\
&\qquad+\frac{1}{n}\left(\phi^\iota_{{q_0}}(w_{0}) +\tilde{C}-e^{-\frac{\pi}{c}}\rho_n \right)+o\left(\frac{1}{n}\right)
\end{align*}

\end{proof}

\begin{lem}\label{prop:eggb}
	For $k_n \leq j \leq M_n - \ell_n$, we have $W_j \in \mathcal{R}_{\eps_j}$ and
	$$W_{j} = W_{k_n} + \sum_{i=k_n}^{j-1}\eps_{i}+\tilde{A}(\eps_{i},w_i) $$
\end{lem}

\begin{proof}
	We prove this by induction on $j$.
	\begin{itemize}
		\item 	Initialization: it comes from the fact that $W_{k_n} = \frac{k_n}{n}+o(\frac{k_n}{n})$
		(Lemma \ref{lem:init} and Lemma \ref{prop:error1}).
		\item Heredity: it follows immediately from Proposition \ref{prop:almosttransv2} and the computation above. 
		
	\end{itemize}	
\end{proof}

\subsection{Exiting the eggbeater}

\begin{lem}[Comparison with outgoing Fatou coordinates]\label{lem:exit}
	We have $w_{M_n-\ell_n} \in \mathcal{P}^o_R$, and
	$$\frac{1}{\eps_{M_n-\ell_n}}\left(\Phi_{\eps_{M_n-\ell_n}}(w_{M_n-\ell_n})-\frac{\pi}{c}\right)=\phi^o_{q_0}(w_{M_n-\ell_n})+{e^{-\frac{\pi}{c}}\frac{\ell_n^2}{2n}} +(1- b_{0,3}) \ln n + E^o + o(1)$$
	where 
	$E^o:=(1- b_{0,3})  \left(\frac{\pi}{c}-\ln c -e^{\frac{\pi}{c}} \int_{\pi/2c}^{\pi/c} e^{-u} \ln \sin(cu) du   \right)$.
\end{lem}

\begin{proof} By Lemma \ref{lem:sumeps} and Lemma \ref{prop:eggb} we know that $W_{M_n-\ell_N} \in \mathcal{R}_{\eps_{M_n-\ell_N}}$ and that  $W_{M_n-\ell_N}=\frac{\pi}{c}-e^{-\frac{\pi}{c}}\frac{\ell_n}{n}+{\frac{k_n^2}{2n^2}+o(\frac{1}{n})}$. Since $w_{M_n-\ell_n}=-c\eps_{M_n-\ell}\cot(cW_{M_n-\ell_n})+O\left(\frac{1}{n}\right)$ we have $w_{M_n-\ell_n}\sim \frac{1}{\ell_n}$ hence for all sufficiently large $n$ we have $w_{M_n-\ell_n} \in \mathcal{P}^o_R$. By the same computation as in the incoming case, we have 
\begin{align*}	
	\psi_{\eps_{M_n-\ell_n}}^o(w_{M_n-\ell_n})&=  -\frac{\eps_{M_n-\ell_n}}{w_{M_n-\ell_n}}+\frac{\eps_{M_n-\ell_n}^2}{2w_{M_n-\ell_n}^2}+o(\eps_{M_n-\ell_n})\\
	&=\eps_{M_n-\ell_n}\left(-\frac{1}{w_{M_n-\ell_n}}+e^{-\frac{\pi}{c}}\frac{\ell_n^2}{2n}+o(1)\right)
	\end{align*}
Recall that $\phi^{o}_{q_0}(w)=-\frac{1}{w}+(1- b_{0,3})\log(w) + o(1)$.	
	
	Next, we have:
	\begin{align*}
	\Phi_{\eps_{M_n-\ell_n}}(w_{M_n-\ell_n}) &= \chi_{\eps_{M_n-\ell_n}} \circ \psi_{\eps_{M_n-\ell_n}}^\iota(w_{M_n-\ell_n})\\
	 &=\chi_{\eps_{M_n-\ell_n}} \left(\psi_{\eps_{M_n-\ell_n}}^o(w_{M_n-\ell_n})+\frac{\pi}{c}\right) \\
	&=\frac{\pi}{c}+W - c \eps_{M_n-\ell_n} (1- b_{0,3}) e^{W+\frac{\pi}{c}} F_c\left(W+\frac{\pi}{c}\right) 
	\end{align*}
where  $W:= \psi_{\eps_{M_n-\ell_n}}^o(w_{M_n-\ell_n})$, and by Lemma \ref{lem:DLFc}
\begin{align*}
 ce^{W+\frac{\pi}{c}} F_c(W+\frac{\pi}{c}) &=\log\left(\frac{\eps_{M_n-\ell_n}}{w_{M_n-\ell_n}}\right)
			+ \ln c + e^{\frac{\pi}{c}} \int_{\frac{\pi}{2c}}^{\frac{\pi}{c}} e^{-u} \ln \sin(cu) du     + o(1)\\
			&=-\log{w_{M_n-\ell_n}}-\ln{(e^{\frac{\pi}{c}}n)}
			+ \ln c + e^{\frac{\pi}{c}} \int_{\frac{\pi}{2c}}^{\frac{\pi}{c}} e^{-u} \ln \sin(cu) du     + o(1)\\
	\end{align*}

Putting all together we obtain
		\begin{align*}
		&\frac{1}{\eps_{M_n-\ell_n}} \left(\Phi_{\eps_{M_n-\ell_n}}(w_{M_n-\ell_n})-\frac{\pi}{c}\right)=\\
		&= -\frac{1}{w_{M_n-\ell_n}} +(1- b_{0,3})\log{w_{M_n-\ell_n}} +e^{-\frac{\pi}{c}}\frac{\ell_n^2}{2n}+(1- b_{0,3}) \ln n +E^o  + o(1)\\
		&=\phi^o_{q_0}(w_{M_n-\ell_n}) +e^{-\frac{\pi}{c}}\frac{\ell_n^2}{2n} +(1- b_{0,3}) \ln n+ E^o  + o(1)
 	\end{align*}

\end{proof}

\begin{lem}\label{lem:compout}
	We have
	$$\phi^o_{q_0}(w_{M_n}) = \phi^o_{q_0}(w_{M_n-\ell_n}) + \ell_n + o(1)$$
\end{lem}
\begin{proof} Recall that $w_{M_n-\ell_n}= O\left(\frac{1}{\ell_n}\right).$ 
	For $M_n-\ell_n\leq j\leq M_n$ we have
	\begin{align*}
		\phi^o_{q_0}(w_{j+1})&=\phi^o_{q_0}({q_0}(w_j)+ b\eps_j^2 + O(\eps_j^3)) \\
		&=\phi^o_{q_0}(w_j)+1+O((\phi^o_{q_0})'(w_j) \eps_j^2 ) \\
		&=\phi^o_{q_0}(w_j)+1 + O\left(\frac{\eps_j^2}{w_j^2} \right) \\
		&=\phi^o_{q_0}(w_j)+1+ O\left(\frac{\ell_n^2}{n^2} \right)
	\end{align*}
	Therefore by induction, $\phi^o_{q_0}(w_{M_n})=\phi^o_{q_0}(w_{M_n-\ell_n})+\ell_n + O\left(\frac{\ell_n^3}{n^2}\right)$,
	and $\frac{\ell_n^3}{n^2}=o(1)$ since $\ell_n\sim n^{\nu}$ with $\nu\in\left(\frac{1}{2},\frac{2}{3}\right)$. 
\end{proof}

\subsection{Conclusion}
Finally we can prove the Main Theorem. We state here a technical, equivalent formulation:
\begin{thm}\label{th:maintech}
	We have
	\begin{align*}
	w_{M_n} &= (\phi^o_{q_0})^{-1}\left( e^{\frac{\pi}{c}}\phi^\iota_{q_0}(w_{0})-\left(e^{\frac{\pi}{c}}-1\right)\phi^{\iota}_p(z_0)-\rho_n +\Gamma\right)+ o(1)
\end{align*}
where 
\begin{multline}\label{eq:Gamma}
	\Gamma:=(e^{\frac{\pi}{c}}-1)\left(\frac{a- b_{0,3}+b_{3,0}}{2b}+a+\frac{1}{2}(1- b_{0,3})+( b_{0,3}-1)\ln{c} \right)
	+( b_{0,3}-a)\frac{\pi}{c}
\\	+	e^{\frac{\pi}{c}}(1- b_{0,3}) \int_0^{\frac{\pi}{c}} e^{-u} \ln \sin(cu) du  
\end{multline}
\end{thm}

\begin{proof} We have:
\begin{align*}
	\phi^o_{q_0}(w_{M_n}) &=  \phi^o_{q_0}(w_{M_n-\ell_n}) + \ell_n + o(1)\\
	&=\frac{1}{\eps_{M_n-\ell_n}}\left(\Phi_{\eps_{M_n-\ell_n}}(w_{M_n-\ell_n})-\frac{\pi}{c}\right)-e^{-\frac{\pi}{c}}\frac{\ell_n^2}{2n}-(1- b_{0,3}) \ln n - E^o +\ell_n + o(1)\\
&=e^\frac{\pi}{c}\phi^\iota_{q_0}(w_{0})-\rho_n+e^\frac{\pi}{c}\tilde{C} - E^o  +o(1).
\end{align*}
where the first equality follows from Lemma \ref{lem:compout}, the second equality follows from  Lemma \ref{lem:exit} and the last equality follows from Lemma \ref{lem:sumeps} and Lemma \ref{prop:eggb}. Note that in this computation we used the fact that $\frac{1}{2n}(e^\frac{\pi}{c}k_n^2- e^{-\frac{\pi}{c}}\ell_n^2)=o(1)$.

Finally recall that
  \begin{align*}\Theta&=b_{0,3} +\frac{a-b_{0,3}+b_{3,0}}{2b},\\
  E^\iota&=( b_{0,3}-1)  \left(\ln c - \int_0^{\frac{\pi}{2c}} e^{-u} \ln \sin(cu) du   \right),\\
  \tilde{C}&=(1-a)e^{-\frac{\pi}{c}}\frac{\pi}{c} +(1-e^{-\frac{\pi}{c}})\left(\Theta+\frac{3}{2}(1- b_{0,3})+(a-1)-\phi^{\iota}_p(z_0)\right)+	E^\iota,\\
  E^o&=(1- b_{0,3})  \left(\frac{\pi}{c}-\ln c -e^{\frac{\pi}{c}} \int_{\frac{\pi}{2c}}^{\frac{\pi}{c}} e^{-u} \ln \sin(cu) du   \right).
  \end{align*}
A quick computation now gives
$$e^\frac{\pi}{c}\tilde{C} - E^o=-\left(e^{\frac{\pi}{c}}-1\right)\phi^{\iota}_p(z_0)+\Gamma,$$
hence
$$\phi^o_{q_0}(w_{M_n})= e^{\frac{\pi}{c}}\phi^\iota_{q_0}(w_{0})-\left(e^{\frac{\pi}{c}}-1\right)\phi^{\iota}_p(z_0)-\rho_n +\Gamma+ o(1).$$

\end{proof}

\begin{rem}\label{rem:general} Note that Theorem \ref{th:maintech} has been proved under the assumption that $\beta_0\in\R$. Following essentially the same proof in the case where $\beta_0\in\C$ (only replacing  the definition of $M_n$ and $\rho_n$ in Definition \ref{defi:mn} by $M_n:= \lfloor(\alpha_0-1)n+ \re(\beta_0)\ln n \rfloor$ and $\rho_n:=\{(\alpha_0-1)n+ \re(\beta_0)\ln n  \}$ ), one could prove that
	\begin{align*}
	w_{M_n} &= (\phi^o_{q_0})^{-1}\left( e^{\frac{\pi}{c}}\phi^\iota_{q_0}(w_{0})-\left(e^{\frac{\pi}{c}}-1\right)\phi^{\iota}_p(z_0)-\rho_n +\Gamma+i\im(b_{0,3}-a)\ln n\right)+ o(1).
	\end{align*}
	It then seems likely that $(z_{n+M_n}, w_{M_n})$  belongs to one of the two parabolic domains $U^\pm$
	from Theorem \ref{th:parbas}, which in turn would imply that $(z_n,w)$ belongs to the parabolic basin of $(0,0)$ for all $n$ large enough. This also seems to be supported by numerical experiments.
\end{rem}

\begin{proof}[Proof of the Main Theorem from Theorem \ref{th:maintech}]
	It only remains to rephrase Theorem \ref{th:maintech} in terms of admissible sequences.
	Let $(n_k)_{k \geq 0}$ be an $(\alpha_0,\beta_0)$-admissible sequence.
	By definition of $M_n$ and $\rho_n$, we have
	$$M_{n_k}=\lfloor (\alpha_0-1) n_k + \beta_0 \ln n_k\rfloor,$$ and 
	$\rho_{n_k}=\{(\alpha_0-1) n_k + \beta_0 \ln n_k \}$. Therefore, by definition of an $(\alpha_0,\beta_0)$-admissible sequence,
	there exists a bounded sequence of integers $(m_k)_{k \geq 0}$ such that 
	$$n_{k+1}-n_k = M_{n_k}+m_k,$$
	and the phase sequence of $(n_k)_{k \geq 0}$ is given by
	\begin{align*}
		\sigma_{k}= n_{k+1}-\alpha_0 n_k-\beta_0 \ln n_k &= n_{k+1}- (M_{n_k}+n_k+\rho_{n_k})\\
		&=m_k-\rho_{n_k}.
	\end{align*}
	By Theorem \ref{th:maintech}, we have 
	$$P^{M_{n_k}}(p^{n_k}(z),w)=\left(p^{n_k+M_{n_k}}(z), \lcal(\alpha_0, \Gamma-\rho_{n_k}; z,w)  \right) + o(1)$$
	and therefore, by the functional equation satisfied by $\lcal$,
	\begin{align*}
	P^{n_{k+1}-n_k}(p^{n_k}(z),w)&=P^{M_{n_k}+m_k}(p^{n_k}(z),w)\\
	&=\left(p^{n_k+M_{n_k}+m_k}(z), \lcal(\alpha_0, \Gamma+m_k-\rho_{n_k}; z,w)  \right) + o(1)\\
	&=\left(p^{n_{k+1}}(z), \lcal(\alpha_0, \Gamma+\sigma_{k}; z,w)  \right) + o(1)
	\end{align*}
	which is the desired result.
\end{proof}

\section{Wandering domains of rank 1}\label{proofwd}

The aim of this section is to prove Theorem \ref{prop:wd}.

\begin{proof}[Proof of Theorem \ref{prop:wd}]
	By our assumption, if $(\sigma_k)_{k \in \N}$ denotes the phase sequence associated to the $(\alpha_0,\beta_0)$-admissible sequence $(n_k)_{k \in \N}$, then 
	$\sigma_k:=n_{k+1}-\alpha_0 n_k -\beta_0\ln{n_k}\xrightarrow{k \to +\infty}\theta$, and hence by the Main Theorem we have $P^{n_{k+1}-n_k}(p^{n_k}(z),w)\xrightarrow{ k\rightarrow \infty}  \left(0, \lcal_{z}(w)\right)$  where  $\lcal_{z}(w):=\lcal(\alpha_0, \Gamma+\theta; z,w)$. 
	
	\medskip 
	
	Let  $\mathcal{E}(W):=\phi^\iota_{q_0}\circ (\phi^{o}_{q_0})^{-1}(W)$ be the \emph{lifted horn map} of $q_0$. The map $\mathcal{E}$ is defined on the open set $\mathcal{U}_{q_0}:=\left(\psi_{q_0}^o\right)^{-1}(\bcal_{q_0})$, which has at least two connected components, one containing an upper half-plane and the other containing a lower half-plane. Moreover, it commutes with the translation by $1$: for all $W \in \mathcal{U}_{q_0}$, 
	 $\mathcal{E}(W+1)=\mathcal{E}(W)+1$.

 Let us define $\sigma:=\Gamma + \theta$, where $\Gamma$ is the constant from the Main Theorem, and
	$$ 
	\tilde \horn_{Z,\sigma}(W):=\alpha_0 \mathcal{E}(W)+(1-\alpha_0 )Z+\sigma
	$$
	as in Definition \ref{defi:hornP}
	
	\begin{lem}\label{lem:super}
		There exists a point $(z_0,w_0)\in\mathcal{B}_p\times \mathcal{B}_{q_0}$ such that $w_0$ is a super-attracting fixed point of the map $\mathcal{L}_{z_0}(w)$. 
	\end{lem}
	\begin{proof} First observe that  $\mathcal{L}_z$ is semi-conjugate to $\tilde \horn_{Z,\sigma}$, where $Z:=\phi_p^\iota(z)$. Indeed, we have $\mathcal{L}_z\circ (\phi^{o}_{q_0})^{-1} =(\phi^{o}_{q_0})^{-1}\circ \tilde \horn_{Z,\sigma}$, hence it suffices to prove that the map $\tilde \horn_{Z,\sigma}$ has a super-attracting fixed point for appropriate choice of $Z \in \C$.
		
		Let $W_0$ be a critical point of $\mathcal{E}$ and observe that since $\mathcal{E}$ commutes with the translation by $1$, it follows that for every $N\in\mathbb{N}$ the point $W_0+N$ is also a critical point of $\mathcal{E}$. 
		
		Next, observe that 
		$$
		\frac{\alpha_0\mathcal{E}(W_0+N)-(W_0+N)+\sigma}{\alpha_0 -1}=\frac{\alpha_0\mathcal{E}(W_0)-W_0+\sigma}{\alpha_0 -1}+N,
		$$ 
		hence for sufficiently large $N_0\in\mathbb{N}$ there exists $z_0\in\mathcal{B}_p$ such that 
		$$
		Z_0:=\phi^\iota_p(z_0)=\frac{\alpha_0\mathcal{E}(W_0+N_0)-(W_0+N_0)+\sigma}{\alpha_0 -1}.
		$$
		It is then straightforward to check that $W_0+N_0$ is a super-attracting fixed point of $\tilde \horn_{Z_0,\sigma}(W)$.
		
	\end{proof}

	Let $(z_0,w_0)\in\mathcal{B}_p\times \mathcal{B}_{q_0}$ such that $w_0$ is a super-attracting fixed point of $\lcal_{z_0}(w)$.
	Let $\mathcal{A}:=\{(z,w)\in\mathcal{B}_p\times \mathcal{B}_{q_0}\mid \mathcal{L}_z(w)=w\}$. The analytic set  $\mathcal{A}$ has pure dimension $1$, and since $w_0$ is a super-attracting fixed point of $\lcal_{z_0}(w)$, the Implicit Function Theorem implies that the point $(z_0,w_0)$ is contained in a regular part of  $\mathcal{A}$. Therefore, there exists a small disk $\Delta_{z_0}$ centered at $z_0$ and a holomorphic function $\eta:\Delta_{z_0}\rightarrow \mathcal{B}_{q_0}$ that satisfies $\eta(z_0)=w_0$ and $h(\Delta_{z_0})\subset\mathcal{A}$ where $h(z):=(z,\eta(z))$. Moreover by restricting that disk if necessary,  we can assume that  $|\lcal_{z}'(\eta(z))|<\frac{1}{2}$ on $\Delta_{z_0}$.
	\medskip
	
	\begin{lem}\label{lem:etaopen}
		The map $\eta: \Delta_{z_0} \to \C$ is non-constant.
	\end{lem}
	
	\begin{proof}
		Recall that we constructed $Z_0,W_0 \in \C$ such that $\tilde \horn_{Z_0,\sigma}(W_0)=W_0$, and $Z_0=\phi_p^\iota(z_0)$,
		$\eta(z_0)=(\phi_{q_0}^o)^{-1}(W_0)$. Again by the Implicit Function Theorem, there exists a holomorphic map
		$\tilde \eta : \Delta_{Z_0} \to \C$ such that $\tilde \eta(Z)$ is a fixed point of $\tilde \horn_{Z,\sigma}$ for all $Z \in \Delta_{Z_0}$,
		where $\Delta_{Z_0}$ is a small disk centered at $Z_0$.
		Moreover, we have $\eta =(\phi_{q_0}^o)^{-1} \circ \tilde \eta \circ \phi_p^\iota$.   
		From the expression of $\tilde \horn_{Z,\sigma}$, it is not difficult to find that $\tilde \eta'(Z_0)=1-\alpha_0 \neq 0$,
		therefore $\tilde \eta$ and also $\eta$ are non-constant.
	\end{proof}

	By the Main Theorem, for each $z\in \Delta_{z_0}$ there exist a disk $D_{z}\subset\mathcal{B}_{q_0}$ centered at $\eta(z)$ and $k_0>0$ such that 
	\begin{equation}\label{eq:proj}\text{proj}_2(P^{n_{k+1}-n_k}(p^{n_{k}}(z)\times D_{z}))\Subset D_{z}
	\end{equation}
	for all $k\geq k_0$, where $\text{proj}_2: \C^2 \to \C$ denotes the projection on the second coordinate. 
	Moreover, we can find a continuously varying family of disks $\{z\}\times D_{z}\subset\mathcal{B}_{p}\times \mathcal{B}_{q_0} $ and a uniform constant $k_0$ with respect to the parameter $z\in \Delta_{z_0}$ for which \eqref{eq:proj} holds. Let us define an open set 
	\begin{equation}\label{j}
	V:=\bigcup_{z\in \Delta_{z_0}} \{p^{n_{k_0}}(z)\}\times D_{z}.
	\end{equation}
	and let $U$ be a connected component of the open set $P^{-n_{k_0}}(V)$  containing  a point $(z_0,w')$ for which $P^{n_{k_0}}(z_0,w')=(p^{n_{k_0}}(z_0),w_0)$. Observe that by the Main Theorem,  the sequence $(P^{n_{k}})_{k \geq 0}$   converges uniformly on compacts in $U$ to a holomorphic map $\varphi(z,w):=(0,\eta(z))$ where $\eta$ is as above. Moreover, since $P$ is a skew-product, this implies that the sequence of iterates $(P^n)_{n \geq 0}$ is bounded on $U$ and therefore that  $U$ is contained in some Fatou component $\Omega \subset \C^2$.

	\begin{lem}
		The map $\eta$ extends holomorphically to a map $\eta: \mathrm{proj}_1(\Omega) \to \bcal_{q_0}$, and
		there exists a subsequence  $(P^{m_{k}})_{k \geq 0}$ that converges locally uniformly on $\Omega$ to the map
		$\Phi:\Omega\rightarrow \{0\}\times\mathcal{B}_{q_0}$  defined by $\Phi(z,w)=(0,\eta(z))$.		
	\end{lem}

	\begin{proof}
		Since $(P^{n_{ k}})$ is normal on $\Omega$, we know that it has a convergent subsequence, let us denote it by $(P^{m_k})$. Moreover since $P$ is a skew-product we know that $\Omega \subset \mathcal{B}_{p}\times \mathbb{C}$ and therefore any limit function of a convergent subsequence of $(P^{n_{ k}})$ must be of the form  $\Phi(z,w)=(0,\kappa(z,w))$,
		and $\kappa(z,w)=\eta(z)$  for all $(z,w) \in U$. By the identity principle, we therefore have $\frac{\partial \kappa}{\partial w}=0$ on $\Omega$, and so $\kappa$ gives a holomorphic continuation of $\eta$ on $\mathrm{proj}_1(\Omega)$, which we still denote by $\eta$.
		Finally, let us argue that $\eta:\text{proj}_1(\Omega) \rightarrow \mathcal{B}_{q_0}$.

		First, observe that if $(z,w) \in \Omega$, then any $\omega$-limit point of $(z,w)$ has bounded orbit under $P$.
		This implies that $\eta$ takes values in $K(q_0)$, the filled-in Julia set of $q_0$.
		Moreover, by Lemma \ref{lem:etaopen}, $\eta$ is non-constant and therefore open; and by definition, 
		$\eta(\Delta_{z_0}) \subset \bcal_{q_0}$.
		So $\eta$ must take values in $\bcal_{q_0}$.

	\end{proof}
	
	Since $\mathcal{E}(W)=W-\pi i(1-b_{0,3})+o(1)$ as $|\im(W)|\rightarrow +\infty$ (see \cite{ABDPR}, Appendix), we have 
	\begin{equation}
		\tilde \horn_{Z,\sigma}(W)=\alpha_0 W+(1-\alpha_0)Z+C+o(1) \quad \text{ as } |\im(W)|\rightarrow +\infty
	\end{equation}
	for some constant $C\in\mathbb{C}$.
	Let $\tilde \horn_\sigma(Z,W):=(Z,\tilde \horn_{Z,\sigma}(W))$ be the lifted horn map of $P$, with the notations of the introduction, and recall that it commutes with the map $T(Z,W)=(Z+1,W+1)$. This map is well defined on $\mathbb{C}\times \mathcal{U}_{q_0}$. The set of fixed points of $\tilde \horn_\sigma$ can be explicitly written as 
	\begin{equation}\label{eq:fixf}
		\text{Fix}_{\tilde \horn_\sigma}:=\left\{\left(\frac{\alpha_0 \mathcal{E}(W)-W}{\alpha_0-1}+\frac{\sigma}{\alpha_0-1},W\right)\mid W\in \mathcal{U}_{q_0}\right\}.
	\end{equation}

	Let us define $\psi(Z,W)=(Z-W, e^{2\pi i W})$  and $\tilde{U}:=\psi(\mathbb{C}\times\mathcal{U}_{q_0})\subset\mathbb{C}\times \mathbb{C}^*$. Observe that there is a small punctured disk $\Delta^*$ such that $\mathbb{C}\times \Delta^*\subset \tilde{U}$ and that there exists a holomorphic map $\Psi:\tilde{U}\rightarrow \mathbb{C}\times \mathbb{C}^*$ such that $\Psi\circ\psi=\psi\circ \tilde \horn_\sigma$. (This map $\Psi$ is holomorphically conjugated to the horn map $H_\sigma$ of $P$, 
	see Definition \ref{defi:hornP}). 
	It extends holomorphically over $\mathbb{C}\times \{0\}$  with 
	$
	\Phi(X,0)=(\alpha_0 X +\alpha_0\pi i(1-b_{0,3})-\sigma,0).
	$
	We still denote by $\Psi$ this extended map.
	
	\begin{lem} $\Omega$ is a wandering domain. 
	\end{lem}
	\begin{proof} Let $\Phi(z,w)=(0,\eta(z))$ be the limit function as in lemma above and define $\Lambda:=\textrm{proj}_1(\Omega)\subset\mathcal{B}_p$. Observe that $\eta(\Lambda)=\textrm{proj}_2( \Phi (\Omega))$ and that $\Sigma:=\{(z,\eta(z))\mid z\in \Lambda\}$ is connected.

		Let  $\text{Fix}_{\Psi}$ be the analytic variety of fixed points of $\Psi$ and observe that   $\text{Fix}_{\Psi}$ is closed in the domain of definition of $\Psi$. Moreover, observe that
		\begin{equation}\label{eq:fixpsi1}
			\psi\left(\frac{\alpha_0 \mathcal{E}(W)-W}{\alpha_0-1}+\frac{\sigma}{\alpha_0-1},W\right)=\left(\frac{\alpha_0 (\mathcal{E}(W)-W)}{\alpha_0-1}+\frac{\sigma}{\alpha_0-1},e^{2\pi iW}\right),
		\end{equation}
	
		and hence
		\begin{equation}\label{eq:fixpsi2}
			\psi\left(\frac{\alpha_0 \mathcal{E}(W)-W}{\alpha_0-1}+\frac{\sigma}{\alpha_0-1},W\right)\xrightarrow{\im(W)\rightarrow + \infty}\left( \frac{-\alpha_0\pi i(1-b)+\sigma}{\alpha_0 -1},0\right).
		\end{equation}
 Since $\text{Fix}_{\Psi}$ is closed, it follows that $\left( \frac{-\alpha_0\pi i(1-b)+\sigma}{\alpha_0 -1},0 \right)\in\text{Fix}_{\Psi}$.

		Let $B_z(w):=\alpha_0 \phi_{q_0}^\iota(w)+(1-\alpha_0) \phi_p^\iota(z)+\sigma$.
		Observe that $\lcal_z = (\phi_{q_0}^o)^{-1} \circ B_z$, and that if $Z:=\phi_p^\iota(z)$, then
		$\tilde \horn_{Z,\sigma}=B_z \circ  (\phi_{q_0}^o)^{-1}$. In other words, $B_z$ also semi-conjugates $\lcal_z$ 
		and $\tilde \horn_{Z,\sigma}$. We let 
		$\Xi(z,w):=(\phi_p^\iota(z), B_z(w))$, and let 
		$$\Sigma':=\Xi(\Sigma)\subset \text{Fix}_{\tilde \horn_\sigma}$$
		be the "lift" of $\Sigma$. Since $\Xi$ is continuous and $\Sigma$ is connected, so is $\Sigma'$.
\medskip		  
		 		 
		Let us assume that $\Omega$ is not wandering. Up to replacing  $\Omega$ with $P^{\ell}(\Omega)$ we may assume that it is periodic, i.e.  $P^m(\Omega)=\Omega$. Observe that this implies that $\Sigma'$ is forward invariant under the translation $T^m$.  Let $\gamma:I\rightarrow \Sigma'$ be a smooth curve such that $\gamma(0)=(Z_0,W_0)$ and $\gamma(1)=(Z_0+m,W_0+m)$
		(this is possible since $\Sigma'$ is connected), and such that $\psi(\gamma(I))$ is a Jordan curve.

		Now observe that by \eqref{eq:fixf}, $\mathrm{Fix}_{\tilde \horn}$ is a holomorphic graph above $\mathcal{U}_{q_0}$
		and therefore is conformally equivalent to an upper half-plane; and by \eqref{eq:fixpsi1} and \eqref{eq:fixpsi2},
		its image under $\psi$ is conformally equivalent to a punctured disk. After the addition of the 
		fixed point $\left( \frac{-\alpha_0\pi i(1-b)+\sigma}{\alpha_0 -1},0 \right)$, we therefore see that		
		$\mathrm{Fix}_\Psi$ is conformally equivalent to a disk.
		The curve $\psi(\gamma)$ is a Jordan curve around $\left( \frac{-\alpha_0\pi i(1-b)+\sigma}{\alpha_0 -1},0 \right)$
		in that disk. Now let us consider the holomorphic map $\det d\Psi: \mathrm{Fix}_\Psi \to \C$. We have 
		$|\det d\Psi|<1$ on $\psi(\gamma)$, since by construction the eigenvalues of $dF$ on $\gamma$ are 1 and  another one which lies in the unit disk.
		But we have computed that $\det d\Psi\left( \frac{-\alpha_0\pi i(1-b)+\sigma}{\alpha_0 -1},0 \right)=\alpha_0>1$,
		which contradicts the maximum principle.
		
	\end{proof}
	This completes the proof of Theorem \ref{prop:wd}.

\end{proof}

\section{Wandering domains for higher periods}\label{sec:higherperiod}

\subsection{Simply connected hyperbolic components}

In this section we assume that $\alpha_0 \in \N^*$ and $q_0(w)=w+w^2$.
We let $\hat h$ denote the classical horn map of $q_0$, and
recall that 
\begin{equation}
	e^{2i\pi (1-\alpha_0)Z+2i\pi \sigma} \hat h(e^{2i\pi W})^{\alpha_0} = e^{2i\pi \tilde \horn_{Z,\sigma}(W)}
\end{equation}

We let $h:=\hat{h}^{\alpha_0}$ and $\lam:=e^{2i\pi (1-\alpha_0)Z+2i\pi \sigma} \in \C^*$, and consider the family $(h_\lam)_{\lam \in \C^*}$, 
defined by $h_\lam:=\lam h$.
Observe that by the choice of $q_0$, the maps $h_\lam$ have exactly 3 singular values: 
\begin{enumerate}
	\item $0$ and $\infty$, which are asymptotic values that are also superattracting fixed points
	\item one free critical value $v_\lam:=\lam v$, where $v:=e^{2i\pi \phi_{q_0}^\iota(-\frac{1}{2})}$.
\end{enumerate}
In particular, if $h_\lam$ has an attracting cycle different from $0$ and $\infty$, then it must capture $v_\lam$.

\begin{defi}
	A hyperbolic component of period $m$ in the family $(h_\lam)_{\lam \in \C^*}$ is a connected component of the set of $\lam \in \C^*$ 
	such that $h_\lam$ has an attracting cycle of period $m$ different from $0$ and $\infty$.
\end{defi}

Note that by \cite[Theorem E]{astorg2021bifurcation}, hyperbolic components as they are defined here
are also stability components.
In order to prove that the Fatou components that we construct are indeed wandering, 
we will use the following result, which also has intrinsic interest:

\begin{thm}\label{prop:hypcompsc}
	Hyperbolic components in the family $(h_\lam)_{\lam \in \C^*}$ are simply connected.
\end{thm}

\begin{figure}[h!]
	\centering
	\begin{subfigure}{.5\textwidth}
		\centering
		\includegraphics[width=.9\linewidth]{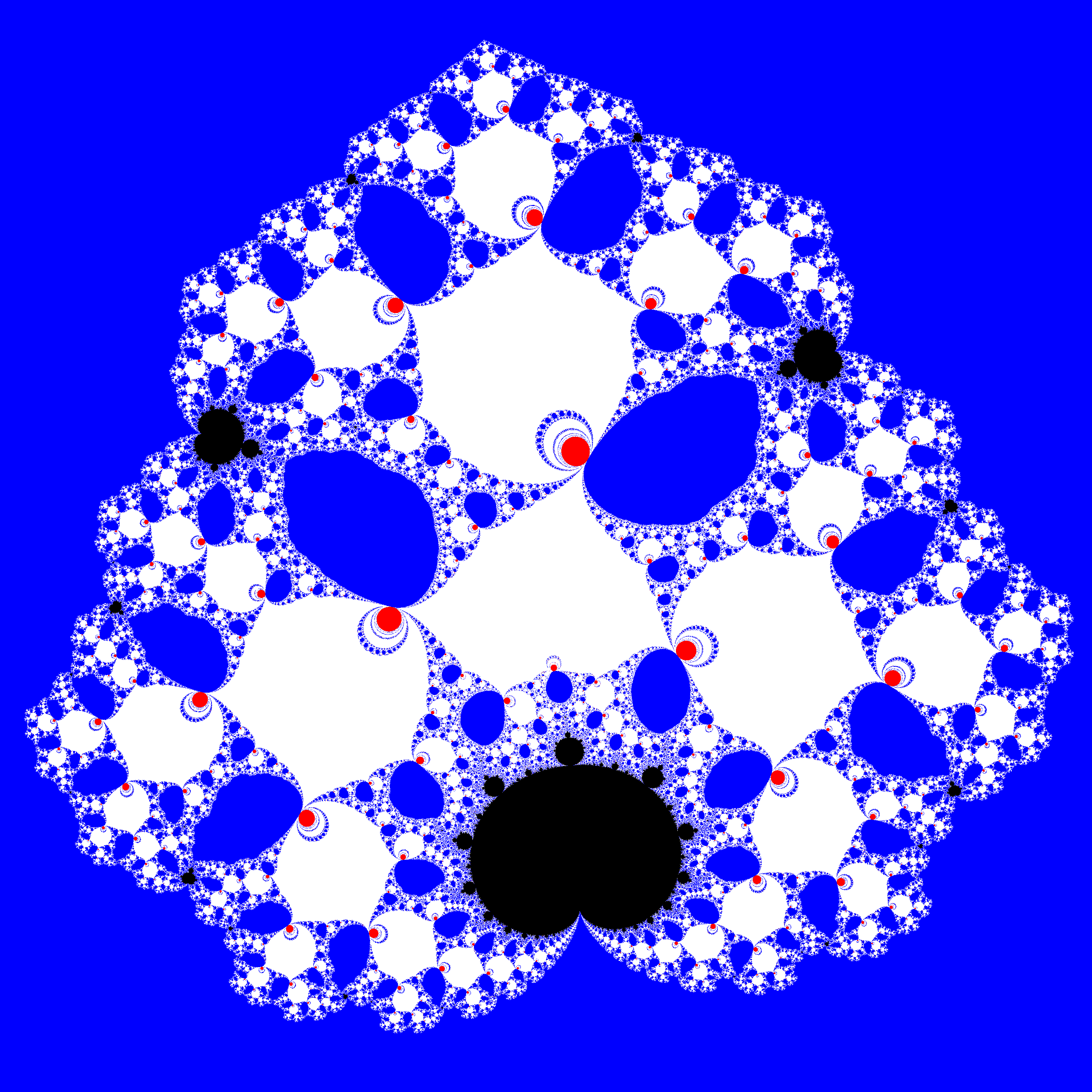}
	\end{subfigure}%
	\begin{subfigure}{.5\textwidth}
		\centering
		\includegraphics[width=.9\linewidth]{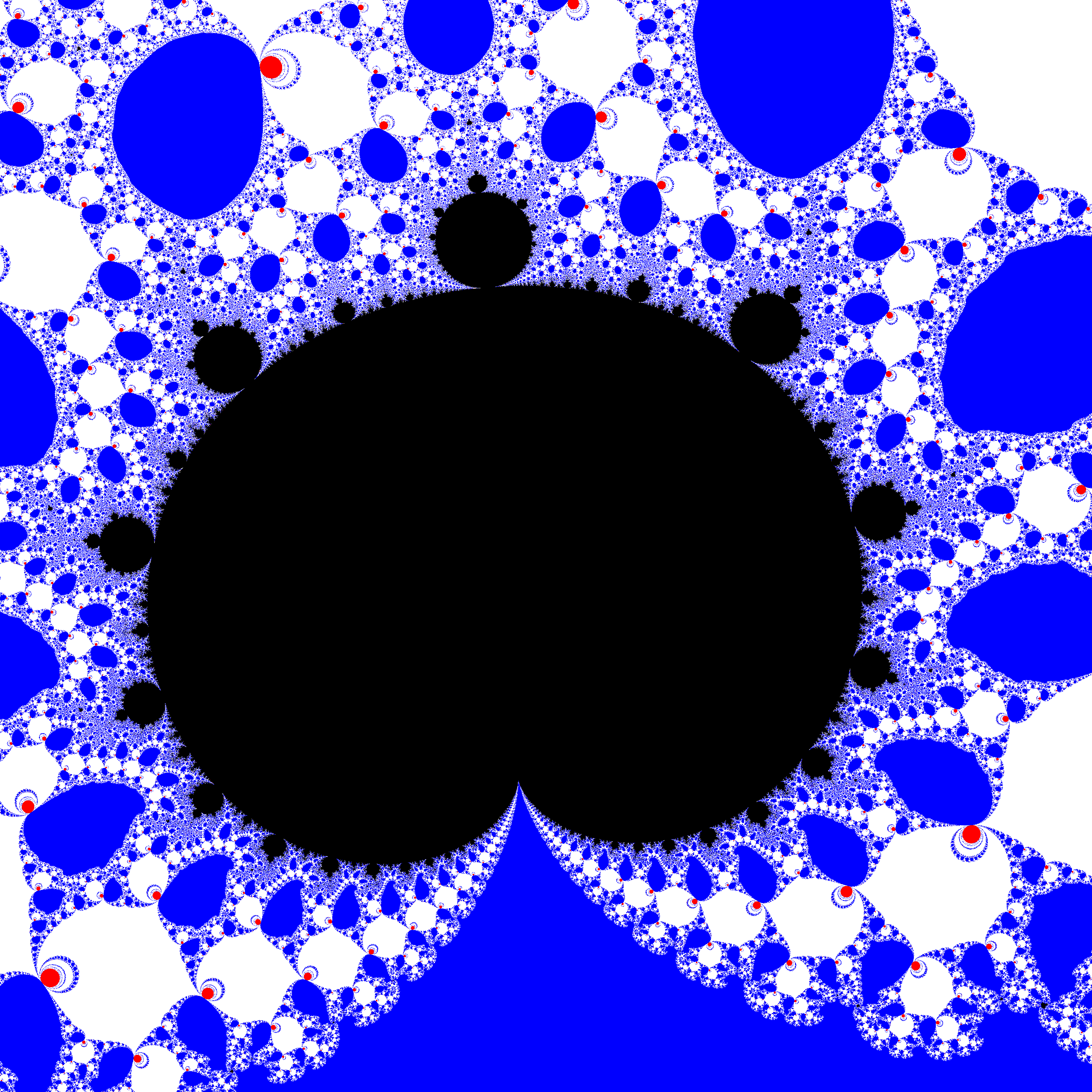}
	\end{subfigure}
	\caption{Parameter space of $(h_\lam)_{\lam \in \C^*}$.
		Hyperbolic components are in black. Red and blue correspond to parameters $\lam$ for which $v_\lam$ is captured by $0$ or $\infty$ respectively, and white to $\lam$ such that $v_\lam$ eventually exits the domain of $h_\lam$. Observe that for all $|\lam|$ large enough, $v_\lam$ is captured by $\infty$ (blue). Right:	a zoom on a copy of the Mandelbrot set (bottom center of the left figure).}
	\label{fig:parhornsquared}
\end{figure}

Before proving Theorem \ref{prop:hypcompsc}, we introduce some further notations:
\begin{defi}
	We let $P_m:=\{(\lam,z) \in \C^* \times \C^*: z=h_\lam^m(z)\}$, and 
	$\tilde \rho: P_m \to \C$ be the map defined by $\tilde \rho(\lam,z)=(h_\lam^m)'(z)$.
\end{defi}

Let $U$ be a hyperbolic component of period $m$ and $\D\subset\C$ the unit disk. Then $U=\mathrm{proj}_1(\Pi)$, where $\Pi$ is a connected component of 
$\tilde \rho^{-1}(\D)$ and $\mathrm{proj}_1: \C^2 \to \C$ is the projection on the first coordinate.
Since for every $\lam \in \C^*$, $h_\lam$ has only one free singular value, it may have at most one 
attracting cycle different from $0$ and $\infty$; therefore if $(\lam, z_1)$ and $(\lam, z_2)$ are in a same fiber of the  map 
$\mathrm{proj}_1: \Pi \to U$, then $z_1$ and $z_2$  must be periodic points of the same attracting cycle.
This means that the function $\tilde \rho: \Pi \to \D$ descends to a well-defined holomorphic function 
$\rho: U \to \D$ satisfying $\tilde \rho = \rho \circ \mathrm{proj}_1$.

\begin{lem}\label{lem:surgery}
	Let $U_0:=U \backslash  \rho^{-1}(\{0\})$.
	The map $\rho : U_0 \to \D^*$ is locally invertible.
\end{lem}

\begin{proof}
	We will prove this using a classical surgery argument, originally due to Douady-Hubbard for 
	the case of the quadratic family (\cite{notesorsay2}).
	Let $\lam_0\in U_0$, and let $V$ be a simply connected open subset of $\D^*$ containing 
	$\rho(\lam_0)$.
	Using a standard surgery procedure, we construct for any $t \in V$
	a quasiconformal homeomorphism $g_t$ such that $g_t \circ h_{\lam_0} \circ g_t^{-1}$ is holomorphic, and $g_t(z_0)$ is a periodic point
	of period $m$ 
	and multiplier $t$. We refer to \cite[Proposition 6.7]{astorg2019wandering},  for the details
	(see also e.g. \cite[Theorem 6.4]{fagellastable2021}).
	
	We let $\phi: V \to \teich(h_{\lam_0})$ be the holomorphic map induced by $t \mapsto \mu_t$, where $\mu_t$ is 
	the Beltrami form associated to $g_t$ and $\teich(h_{\lam_0})$ is the dynamical Teichmüller space of $h_{\lam_0}$.
	For a definition of the dynamical Teichmüller space, see \cite{mcmullen1998quasiconformal}, \cite{astorgteich}.
	Let  $\hat{V} \subset U_0$ be a simply connected domain containing $\lam_0$.
	Since for all $\lam \in \hat{V}$ the free critical value $v_\lam$ remains captured by the attracting 
	cycle, the family $(h_\lam)_{\lam \in \hat{V}}$ is $J$-stable by \cite[Theorem E]{astorg2021bifurcation}.
	In fact, since there are no non-persistent singular relations for the family $(h_\lam)_{\lam \in \hat{V}}$, 
	by \cite[Theorem 7.4]{mcmullen1998quasiconformal} (stated for rational maps, but whose proof carries over verbatim in this setting), the map $h_{\lam_0}$ is in fact \emph{structurally stable} on $\rs$:
	there is a second holomorphic family $\hat{g}_\lam$ of quasiconformal homeomorphisms $\hat{g}_\lam:\rs \to \rs$ such that
	$h_\lam:=\hat{g}_\lam \circ h_{\lam_0} \circ \hat{g}_\lam^{-1}$ for all $\lam \in \hat{V}$, and 
	$\hat{g}_{\lam_0}=\id$.
	
	We let $\hat\phi :\hat{V} \to \teich(h_{\lam_0})$ denote the map induced by $\lam \mapsto \hat{\mu}_\lam$, where 
	$\hat{\mu}_\lam$ is the Beltrami form associated to $\hat{g}_\lam$. Let $\xi:=\frac{d}{d\lam}_{|\lam=\lam_0} \hat{g}_\lam$,
	and observe that since $\hat{g}_\lam(v_{\lam_0})=v_\lam= \lam v$, we have $\xi(v_{\lam_0}) \neq 0$.
	By \cite[Proposition 5]{astorgteich}, the derivative $\hat{\phi}'(\lam_0)$ is therefore non-zero.
	Therefore, up to restricting $V$, we may assume that $\phi(V) \subset \hat{\phi}(\hat{V})$ and that 
	there exists a well-defined inverse branch $\hat{\phi}^{-1}: \phi(V) \to \hat{V}$.
	Let $c: V \to \hat{V}$ be the map defined by $c:=\hat{\phi}^{-1}\circ\phi$.
	Then $c$ is a holomorphic local inverse of $\rho$, which maps $\rho(\lam_0)$ 
	to $\lam_0$; the Lemma is proved.
\end{proof}

\begin{lem}\label{lem:multcov}
	 The map $\rho : U_0 \to \D^*$ is a covering map of finite degree.
\end{lem}

\begin{proof}[Proof of Lemma \ref{lem:multcov}]
%
	
	We start by claiming that $\Pi$ is relatively compact in 
	$P_m$. Indeed, $U=\mathrm{proj}_1(\Pi)$ is relatively compact in $\C^*$
	because if $|\lam|$ is small (respectively large) enough, $v_\lam$ is captured by the super-attracting fixed point $0$ 
	(respectively $\infty$). Moreover, by \cite[Theorem A]{astorg2021bifurcation}, the map $\mathrm{proj}_1: P_m \to \C^*$ is proper,
	because the only two asymptotic values in the family $(h_\lam)_{\lam \in \C^*}$ are persistently fixed.
	Therefore $\Pi$ is relatively compact in $P_m$. Since $\tilde \rho$ is analytic (hence continuous) on $P_m$, 
	and since the set $\Pi_0$ is a connected component of $\tilde \rho^{-1}(\D^*)$, 
	this proves that $\tilde \rho: \Pi_0 \to \D^*$ is proper. Consequently, so is $\rho: U_0 \to \D^*$.

	By Lemma \ref{lem:surgery}, the map $ \rho : U_0 \to \D^*$ is also locally invertible;
	therefore it is a finite degree covering map.	
\end{proof}

\begin{proof}[Proof of Theorem \ref{prop:hypcompsc}]
	By the lemma above, $\rho : U_0 \to \D^*$ is a finite degree covering map. This implies that  there exists 
	$\lam_0 \in U$ such that $U_0=U \backslash \{ \lam_0 \}$, and that $U_0$ is isomorphic to a punctured disk and $U$ to a disk. 
\end{proof}

\subsection{Proof of Theorem \ref{prop:many-wd}}

We state here is a slightly more precise statement of Theorem \ref{prop:many-wd}:

\begin{thm}\label{thm:manyGO}
	To each hyperbolic component $U$ of the family $(h_\lam)_{\lam \in \C^*}$, we can associate a wandering Fatou component $\Omega_U$ of $P$. Moreover, if $U_1 \neq U_2$, then $\Omega_{U_1}$ and $\Omega_{U_2}$ are in different grand orbits of $P$.
\end{thm}

Since $\alpha_0$ is an integer we know already that there exists an $\alpha_0$-admissible sequence $n_k$ and  $\sigma\in \mathbb{C}$, such that $P^{n_{k+1}-n_k}(p^{n_k}(z),w)\rightarrow \lcal(\alpha_0,\sigma; z, w)$ uniformly on compacts in $\mathcal{B}_p\times\mathcal{B}_{q_0}$.
Let $(\lam_0,x_0) \in \C^* \times \C^*$ be such that $x_0$ is a super-attracting periodic point of exact period $m$ 
for $h_{\lam_0}$. Let $(z_0,w_0) \in \bcal_p \times \bcal_{q_0}$ be such that 
$e^{2i\pi (1-\alpha_0) \phi_p^\iota(z_0)}=\lam_0$ and $e^{2i\pi \phi_{q_0}^\iota(w_0)}=x_0$.
Then $w_0$ is an attracting fixed point of $\lcal(\alpha_0,\sigma; z_0, \cdot)$.

Applying the Main Theorem as in Section \ref{proofwd}, we know that $(z_n,w_0)$ belongs to the Fatou set of $P$ for large enough $n$, where 
$z_n:=p^n(z)$. We let $n_0$ be large enough, and $\Omega=\Omega_{(z_{n_0},w_0)}$ denote the Fatou component
containing $(z_{n_0},w_0)$. We also know (again, by applying inductively the Main Theorem as in Section \ref{proofwd}), 
than there is an increasing sequence of integers $(n_k)$ such that $P_{|\Omega}^{n_k-n_0}(z,w) \to (0,\eta(z))$,
where $\eta(z)$ is an attracting periodic point of period $m$ of $\lcal(\alpha_0,\sigma; z,\cdot)$, with $\eta(z_{n_0})=w_0$.

\begin{lem}
	The Fatou component $\Omega_{(z_{n_0},w_0)}$ is wandering. 
\end{lem}

\begin{proof}
	The proof is similar to the one in Section \ref{proofwd}, and we use some of the same notations.
	We assume for a contradiction that $\Omega$ is not wandering: then $P^{k+\ell}(\Omega)=P^k(\Omega)$, for some $k \in \N$ and $\ell \in \N^*$.
	Up to replacing $\Omega$ by $P^k(\Omega)$, we may assume $k=0$.

	There exists some continuous curve joining $(z_{n_0},w_0)$ and $P^{\ell}(z_{n_0},w_0)$ inside $\Omega$.
	Using the convergence of $P^{n_k-n_0}$ to $(0,\eta)$, we obtain a curve joining $\eta(z_{n_0})$ and $\eta(z_{n_0+\ell})$ 
	inside $\Sigma:=\eta(\Lambda)$, where $\Lambda:=\mathrm{proj}_1(\Omega)$.
	We let $\Sigma'$ be as in Section \ref{proofwd}: $\Sigma'$ is an open subset of 
	$\mathrm{Per}_m(\tilde H_\sigma):=\{ (Z,W) \in \C^2: \tilde H_{\sigma}^m(Z,W) \}$: then there is a curve in $\Sigma'$ joining 
	$(Z_0,W_0)$ and $(Z_0+\ell,W_0+\ell)$, where $Z_0:=\phi_p^\iota(z_{n_0})$ and $W_0:=\alpha_0 \phi_{q_0}^\iota(w_0)+(1-\alpha_0) Z_0$.
	
	Finally, we consider the image of this curve under the map 
	\begin{equation}\label{map:e}
	e: (Z,W) \mapsto (e^{2i\pi (1-\alpha_0) Z}, e^{2i\pi W} ).
	\end{equation}
	
	It now becomes a closed loop in $\Pi:=e(\Sigma')$, which we denote by $\gamma:=(\gamma_1, \gamma_2)$.
	By construction, the loop $\gamma_2$ is non-contractible in $\C^*$; however, it is contained in the hyperbolic component $U=\mathrm{proj}_1(\Pi)$. But this contradicts Theorem \ref{prop:hypcompsc}.
\end{proof}

\begin{lem}\label{lem:diffgo}
	If $(z_i, w_i)$ ($i=1,2$) are such that $\lam_i=e^{2i\pi(1-\alpha_0)\phi_p^\iota(z_i)}$ belong to two different hyperbolic components $U_i$, then the wandering Fatou components $\Omega_i:=\Omega_{(z_{i,n_i'},w_i)}$ are not in the same grand orbit.
\end{lem}

\begin{proof}
	The idea of the proof is similar. Let us consider two  wandering Fatou components
	$\Omega_i:=\Omega_{(p^{n_i'}(z_i),w_i)}$
	constructed above.
	Recall that $(z_i,w_i)$ are such that $w_i$ is a super-attracting periodic point 
	of $\lcal(\alpha_0, \sigma; z_i, \cdot)$, and that $(p^{n_i'}(z_i), w_i) \in \Omega_i$.
	Let us denote by $\ell_i$ the periods of $w_i$.

	Assume by contraposition that $\Omega_1$ and $\Omega_2$ are in the same grand orbit of Fatou components for $P$:  
	then there exists $m_i \in \N$ such that $P^{m_1}(\Omega_1)= P^{m_2}(\Omega_2)=:\Omega$.
	Moreover, there is an increasing sequence of integers
	$(n_k)_{k \geq 0}$ such that $P_{|\Omega_i}^{n_k-n_i'}$ converge to the maps $(z,w) \mapsto (0,\eta_i(z))$,
	where for all $z \in \Lambda_i:=\mathrm{proj}_1(\Omega_i)$, $\eta_i(z_i)$
	are periodic points of respective periods $\ell_i$ of the maps $\lcal_{z_i}:=\lcal(\alpha_0,\sigma; z_i,\cdot)$. 
	(This sequence $(n_k)$ is obtained by taking an $\alpha_0$-admissible sequence, and then extracting a subsequence by taking only one term every  $\mathrm{lcm}(\ell_1,\ell_2)$).
	By normality, it is easy to see that the multipliers of those fixed points cannot be repelling: $\rho_i(z):=(\lcal_{z}^{\ell_i})'(\eta_i(z) ) \in \overline{\D}$ for all $z \in \Lambda_i$.
	Since non-constant holomorphic functions are open and $\rho_i(z_i)=0 \in \D$, we must therefore in fact have $\rho_i(z) \in \D$ for all $z \in \Lambda_i$.
	
	Next, we claim that there exists $\xi: \Lambda:= \mathrm{proj}_1(\Omega) \to \C$
	 such that $\eta_i=q_0^{N_i} \circ \xi \circ p^{m_i}$ for some $N_i \in \N$. Indeed, since $P_{|\Omega_i}^{n_k-n_i'} \to (0,\eta_i)$
	 on $\Omega_i$, there exists functions $\xi_i: \Lambda \to \C$ such that 
	$P^{n_k-n_i'-m_i} \to (0,\xi_i)$ on $\Omega$, and 
	$$\eta_i = \xi_i \circ p^{m_i}.$$
	Assume without loss of generality 
	that $N_0:=n_1'+m_1-(n_2'+m_2) \geq 0$. Then
	$$(0,\xi_2)=\lim_k P^{n_k-n_2'-m_2}=\lim_k P^{n_k-n_1'-m_1+N_0}=P^{N_0} \circ (0,\xi_1)$$
	so that $q_0^{N_0} \circ \xi_2=\xi_1$. So we can take $\xi:=\xi_2$, $N_1:=0$ and $N_2:=N_0$.

	Recall now that $\Xi(z,w):=(\phi_p^\iota(z), \alpha_0 \phi_{q_0}^\iota(w)+(1-\alpha_0) \phi_p^\iota(z)+\sigma)$, $\Sigma_i:=\{(z, \eta(z)) : z \in \Lambda_i\}$, and 
	$\Sigma_i':=\Xi(\Sigma_i)$. Let $\gamma=(\gamma_1,\gamma_2):[0,1]\rightarrow\Omega$ be a continuous curve joining 
	$P^{m_1}(p^{n_1'}(z_1), w_1)$ and $P^{m_2}(p^{n_2'}(z_2), w_2)$ in $\Omega$.
	Let $(Z_i, W_i):=\Xi(p^{n_i'}(z_i), w_i)$. Note that for all $k, \ell \in \N$ and $(z,w) \in \bcal_p \times \bcal_{q_0}$,
	$\Xi(p^k(z),q_0^{\ell}(w)) - \Xi(z,w) \in \Z^2$.
	In particular, $\tilde \gamma(t):=\Xi (\gamma_1(t),\xi(\gamma_1(t)))$ gives a continuous curve satisfying
	the following properties:
	\begin{enumerate}
		\item $\tilde \gamma(0) - \Xi (z_1,\eta_1(z_1)) \in \Z^2$
		\item $\tilde \gamma(1) - \Xi (z_2,\eta_2(z_2)) \in \Z^2$
		\item\label{item:attr} for all $t \in [0,1]$, $\tilde \gamma(t) \in \mathrm{Per}_\ell(\tilde H_\sigma)$, 
		where $\ell:=\mathrm{lcm}(\ell_1, \ell_2)$, $\tilde H_\sigma(Z,W)$ is the lifted horn map defined in \eqref{eq:hornP}, 
		and $\left(\frac{\partial }{\partial W}\tilde{H}_{Z,\sigma}^\ell\right)(\tilde \gamma(t))\in \D$.
	\end{enumerate}
	(Property \ref{item:attr} comes from the previous observation that $\rho_i(z) \in \D$ for all $z \in \Lambda_i$.)

	Finally, we consider a curve $(\hat{\gamma}_1, \hat{\gamma}_2):=e\circ \tilde \gamma$, were $e$ is given by equation \eqref{map:e}. Then $\hat{\gamma}_1$ is a continuous curve 
	joining $\lam_1$ and $\lam_2$, inside a hyperbolic component of period (dividing) $\ell$ 
	for the family $(h_\lam)_{\lam \in \C^*}$, which is a contradiction.
	Thus Lemma \ref{lem:diffgo} and Theorem \ref{thm:manyGO} are proved.
\end{proof}

Finally, to obtain Theorem \ref{prop:many-wd} from Theorem \ref{thm:manyGO}, we just need to know that there are countably many hyperbolic components in the family $(h_\lam)_{\lam \in \C^*}$. Since we have proved that the multiplier map 
is a conformal uniformization of any hyperbolic component on the unit disk, it is enough to prove that there are countably many $\lam \in \C^*$ such that $h_\lam$ has a super-attracting periodic point (different from $0$ or $\infty$). 
But this follows from e.g. [\cite{astorg2021bifurcation}, Proposition 5.1].

\section{Admissible sequences and Pisot numbers}\label{section:pisot}

We will give in this section the proof of Theorem \ref{main:pisot}. 

\begin{lem}\label{lem:sequence}For every $(\alpha,\beta)$-admissible sequence $(n_k)_{k\geq 0}$ there exist a real number $\zeta>0$ and a bounded sequence of real numbers $(d_k)_{k\geq 0}$ such that
	$$n_k= \zeta\alpha^k-k\frac{\beta\ln\alpha}{\alpha-1}+d_k, \qquad \forall k\geq 0.$$
	Moreover, if we let $\rho_k:=n_{k+1}-\alpha n_k -k\beta\ln\alpha$, then 
	$$\rho_k=\sigma_k+\beta\ln\zeta +o(1)$$
	and
	$$\zeta=n_0+\frac{\beta\ln\alpha}{(\alpha-1)^2}+\frac{1}{\alpha}\sum_{j=0}^{\infty}\frac{\rho_j}{\alpha^j}$$
	and
	$$d_k=-\frac{\beta\ln\alpha}{(\alpha-1)^2}-\frac{1}{\alpha}\sum_{j=0}^{\infty}\frac{\rho_{j+k}}{\alpha^j}.$$
	
\end{lem}

\begin{proof}
	First we study the asymptotic behaviour of the $(\alpha,\beta)$-admissible sequences.
	\medskip
	
	{\bf Claim 1:} For every $(\alpha,\beta)$-admissible sequence $(n_k)_{k\geq 0}$ there exist constants $\zeta, C\geq 0$ such that $\left|n_k-\zeta\alpha^k\right|\leq Ck$ for all $k>0.$
	
	\begin{proof}[Proof of Claim 1]
		Let us define $u_k:=\frac{n_k}{\alpha^k}$ and observe that $u_k=u_{k-1}+\frac{\beta}{\alpha^k}\ln{u_{k-1}}+\frac{\beta(k-1)\ln\alpha}{\alpha^k}+\frac{\sigma_{k-1}}{\alpha^k}$ where $(\sigma_k)_{k\geq 0}$ denotes the phase sequence of $(n_k)_{k\geq 0}$. Since the sequence of phases is bounded and $\alpha>1$ it is easy to see that there exists $D>0$ such that $|u_{k}|<D k$ for all $k>0$, hence the sequence $u_k$ converges to some positive real number $\zeta$. It follows that
		\begin{align*}
		\zeta= u_0+\sum_{j=0}^{\infty}(u_{j+1}-u_j)=n_0+\sum_{j=0}^{\infty}\frac{\beta}{\alpha^{j+1}}\ln{u_{j}}+\frac{j\beta \ln\alpha}{\alpha^{j+1}}+\frac{\sigma_{j}}{\alpha^{j+1}}
		\end{align*}
		where the sum converges absolutely. Finally observe that 
		$$n_k= \alpha^ku_k=\zeta\alpha^k-\frac{1}{\alpha}\sum_{j=0}^{\infty}\frac{\beta}{\alpha^{j}}\ln{u_{j+k}}+\frac{(j+k)\beta \ln\alpha}{\alpha^{j}}+\frac{\sigma_{j+k}}{\alpha^{j}}
		$$
		and that there exists $C>0$ such that 
		$$\left|\frac{1}{\alpha}\sum_{j=0}^{\infty}\frac{\beta}{\alpha^{j}}\ln{u_{j+k}}+\frac{(j+k)\beta \ln\alpha}{\alpha^{j}}+\frac{\sigma_{j+k}}{\alpha^{j}}\right|<Ck, \qquad \forall k\geq0.$$
		
	\end{proof}
	
	{\bf Claim 2:} We have $\rho_k=\sigma_k+\beta\ln\zeta +o(1)$.
	
	\begin{proof}[Proof of Claim 2] Observe that by the previous lemma, we have
		\begin{align*}
		\rho_k&=n_{k+1}-\alpha n_k -k\beta\ln\alpha\\
		&=\sigma_k+\beta\ln n_k-k\beta\ln\alpha\\
		&=\sigma_k+\beta\ln u_k=\sigma_k+\beta\ln \zeta +o(1)
		\end{align*}
	\end{proof}
	
	{\bf Claim 3:} We have $$\zeta=n_0+\frac{\beta\ln\alpha}{(\alpha-1)^2}+\frac{1}{\alpha}\sum_{j=0}^{\infty}\frac{\rho_j}{\alpha^j},\qquad d_k=-\frac{\beta\ln\alpha}{(\alpha-1)^2}-\frac{1}{\alpha}\sum_{j=0}^{\infty}\frac{\rho_{j+k}}{\alpha^j}.$$

	\begin{proof}[Proof of Claim 3]
		Recall that $\rho_k=\sigma_k+\beta\ln u_k$. From the proof of Claim 1 it now follows that 
		\begin{align*}\zeta&= n_0+\sum_{j=0}^{\infty}\frac{\beta}{\alpha^{j+1}}\ln{u_{j}}+\frac{j\beta \ln\alpha}{\alpha^{j+1}}+\frac{\sigma_{j}}{\alpha^{j+1}}\\
		&=n_0+\frac{1}{\alpha}\sum_{j=0}^{\infty}\frac{j\beta \ln\alpha}{\alpha^{j}}+\frac{\rho_{j}}{\alpha^{j}}\\
		&=n_0+\frac{\beta\ln\alpha}{(\alpha-1)^2}+\frac{1}{\alpha}\sum_{j=0}^{\infty}\frac{\rho_j}{\alpha^j}
		\end{align*}
		and 
		\begin{align*}
		d_k&=k\frac{\beta\ln\alpha}{\alpha-1}+n_k-\zeta\alpha^k\\
		&=k\frac{\beta\ln\alpha}{\alpha-1}-\frac{1}{\alpha}\sum_{j=0}^{\infty}\frac{\beta}{\alpha^{j}}\ln{u_{j+k}}+\frac{(j+k)\beta \ln\alpha}{\alpha^{j}}+\frac{\sigma_{j+k}}{\alpha^{j}}\\
		&=k\frac{\beta\ln\alpha}{\alpha-1}-\frac{1}{\alpha}\sum_{j=0}^{\infty}\frac{(j+k)\beta \ln\alpha}{\alpha^{j}}+\frac{\rho_{j+k}}{\alpha^{j}}\\
		&= -\frac{\beta\ln\alpha}{(\alpha-1)^2}-\frac{1}{\alpha}\sum_{j=0}^{\infty}\frac{\rho_{j+k}}{\alpha^{j}}
		\end{align*}
		
	\end{proof}
	
	This completes the proof of the proposition. 
\end{proof}

\begin{rem}Note that by Lemma \ref{lem:sequence}, the phase sequence $(\sigma_k)_{k\geq0}$  of a $(\alpha,\beta)-$admissible sequence converges to a cycle if and only if the sequence $(\rho_k)_{k\geq0}$ converges to a cycle of the same period. Hence the phase sequence $(\sigma_k)_{k\geq0}$ converges to a cycle if and only if the sequence $(d_k)_{k\geq0}$ converges to a cycle of the same period.
\end{rem}

\begin{coro}\label{coro:pisot}If $(n_k)_{k\geq0}$ is an $\alpha-$admissible sequence whose phase sequence converges to zero, then 
	$\alpha$ has the Pisot property.
\end{coro}
\begin{proof} Since $(n_k)_{k\geq0}$ is an $\alpha-$admissible sequence, $\beta=0$ and $\rho_k=\sigma_k$ (using the notation from Lemma \ref{lem:sequence}). Moreover, since $(\sigma_k)_{k\geq0}$ converges to zero the same holds for the sequence $(d_k)_{k\geq0}$, and hence we have  $\|\zeta\alpha^k\|\rightarrow 0$.
\end{proof}

\begin{lem}\label{lem:reduction}
	Let $(n_k)_{k\geq 0}$ be an $(\alpha,\beta)$-admissible sequence and $(\sigma_k)_{k\geq 0}$ its phase sequence. Then $(\sigma_k)_{k\geq 0}$ converges to a cycle of period $\ell$ if and only if $m_k:=n_{k+\ell}-n_k$ is $\alpha$-admissible sequence and whose phase sequence converges to $\ell\beta \ln \alpha$.  
\end{lem}

\begin{proof}Observe that
	\begin{align*}m_{k+1}-\alpha m_k&=n_{k+1+\ell}-n_{k+1}-\alpha(n_{k+\ell}-n_k)\\
	&=(n_{k+1+\ell}-\alpha n_{k+\ell})-(n_{k+1}-\alpha n_k)\\
	&=\sigma_{\ell+k}-\sigma_k+\beta\ln\frac{n_{k+\ell}}{n_{k}}\\
	&=\sigma_{\ell+k}-\sigma_k +\ell\beta\ln\alpha +o(1).
	\end{align*}
\end{proof}

\begin{coro}\label{coro:phasesquence}
	If  $(n_k)_{k\geq 0}$ is $\alpha$-admissible with converging phase sequences, then $m_k:=n_{k+1}-n_k$ is $\alpha$-admissible and has phase converging to zero.
\end{coro}
\begin{proof}This follows from the previous lemma with $\beta=0$.
\end{proof}

\begin{lem}\label{lem:pisotalpha}
	If $\alpha$ has the Pisot property, then there exists an $\alpha$-admissible sequence whose phase sequence converges to $0$.
\end{lem}

\begin{proof} Since  $\alpha$ has the Pisot property there is $\zeta>0$ such that $\|\zeta \alpha^k\|\to 0$. Now define $n_k:=[\zeta\alpha^k]$ and observe that $n_{k+1}-\alpha n_k=-\|\zeta \alpha^{k+1}\|+\alpha\|\zeta \alpha^k\|\rightarrow 0$.
\end{proof}

\begin{lem}\label{lem:phasesquence}
	Let $(n_k)$, $(m_k)$ be two $\alpha$-admissible sequences, and let $j, j_1,j_2 \in \Z$. Then:
	\begin{enumerate}
		\item $n_{k+j}$ is again an $\alpha$-admissible sequence, and $\sigma(n_{k+j}) = \sigma(n_k)$
		\item if $j_1 n_k  + j_2 m_k$ is strictly increasing, then it is an $\alpha$-admissible sequence, and 
		$\sigma(j_1 n_k + j_2 m_k) = j_1 \sigma(n_k)+j_2 \sigma(m_k)$
		\item if $(m_k)$ is $\alpha$-admissible and $\eps_k \in \ell^\infty$, then 
		$n_k:=m_k + \eps_k$ is $\alpha$-admissible, and $\sigma(n_k) = \sigma(m_k) + \eps_{k+1}-\alpha \eps_k$.
	\end{enumerate}
\end{lem}
\begin{proof}This is a direct computation.
\end{proof}

Observe that Corollary \ref{coro:pisot}, Corollary \ref{coro:phasesquence}, Lemma \ref{lem:pisotalpha} and Lemma \ref{lem:phasesquence} imply the following result which settles claim (1) of Theorem \ref{main:pisot}.

\begin{coro}\label{coro:main}
	Let $\alpha>1$ and $m \in \N^*$ arbitrary.
	The following are equivalent:
	\begin{enumerate}
		\item $\alpha$ has the Pisot property,
		\item there exists an $\alpha$-admissible sequence whose phase sequence converge,
		\item there exists an $\alpha$-admissible sequence whose phase sequence converge to a cycle of exact period $m$.
		
	\end{enumerate}
\end{coro}

Let us mention that for a very special type of $\alpha$-admissible sequences a similar conclusions were already made by Dubickas, see \cite{dubickas2009integer}.

Finally the claim (2) of Theorem \ref{main:pisot} follows from  Lemma \ref{lem:reduction}, Corollary \ref{coro:main} and the following remark.

\begin{rem}\label{rem:final} Let $(n_k)_{k\geq0}$ be an $(\alpha,\beta)$-admissible sequence and denote $\theta=\frac{\beta\ln \alpha}{\alpha-1}$ and $m_k=n_k+\lfloor k\theta\rfloor$. Observe that by Lemma  \ref{lem:sequence} we have
	\begin{align*}
	n_{k+1}-\alpha n_k -\beta\ln n_k &=n_{k+1}-\alpha n_k -k\beta\ln \alpha- \beta\ln \zeta +o(1)\\
	&=m_{k+1}-\alpha m_k+\{(k+1)\theta\}-\alpha\{k\theta\}-\theta- \beta\ln \zeta +o(1).
	\end{align*}
	It follows that the phase sequence of $(n_k)_{k\geq0}$ converges to a cycle if and only if the sequence $(m_k)_{k\geq0}$ is $\alpha$-admissible  and the sequence 
	$m_{k+1}-\alpha m_k+\{(k+1)\theta\}-\alpha\{k\theta\}$ converges to a cycle of the same period as $\sigma(n_k)$. 
	
	Hence if we take $(m_k)_{k\geq0}$ to be an $\alpha$-admissible  the sequence whose phase sequence converges to zero (note that such always exists since $\alpha$ has the Pisot property) and if $\theta=\frac{k_1}{k_2}\in\mathbb{Q}$ then clearly the sequence $n_k:=m_k-\lfloor k\theta\rfloor$ is an $(\alpha,\beta)$-admissible sequence whose phase sequence converges to a cycle of period $k_2$.
\end{rem}
Note that the sequence $(\{(k+1)\theta\}-\alpha\{k\theta\})_{k\geq0}$ is uniformly distributed modulo $1$ if and only if  $\theta$ is an irrational number, therefore it is reasonable to consider the following question.

\begin{qst} Let $\alpha>1$ have the Pisot property. From the previous remark we already know that $\theta\in\mathbb{Q}$ is a sufficient condition for the existence of an $\alpha$-admissible sequence $(n_k)_{k\geq0}$ such that  
	the sequence $n_{k+1}-\alpha n_k+\{(k+1)\theta\}-\alpha\{k\theta\}$ converges to a cycle. Is this condition also necessary?
\end{qst}

%

\section{Proof of Theorem \ref{thm:top}}\label{sec:topinv}

Let $P_1, P_2$ be two skew-products that are topologically conjugated in a neighborhood of the origin, 
that is, there is a homeomorphism $\mathfrak h: U \to V$ with $\mathfrak{h} \circ P_1 = P_2 \circ \mathfrak{h}$ and $U,V$ 
are open neighborhoods of $(0,0)$.
We will assume without loss of generality that $U,V$ are bounded in $\C^2$. The map $\mathfrak h$ is of the form
$$\mathfrak h(z,w)=(\mathfrak f(z), \mathfrak g(z,w)),$$
and $\mathfrak{f}$ conjugates locally $p_1$ to $p_2$: $\mathfrak{f} \circ p_1=p_2 \circ \mathfrak{f}$.  We will write $\mathfrak g_z(w):=\mathfrak g(z,w)$.
We will also denote by $\lcal_i$, $\alpha_i$, $\beta_i$ for $i \in \{1,2\}$ the quantities appearing in 
the Main Theorem, and $(n_k^i)_{k \in \N}$ two $(\alpha_i,\beta_i)$-admissible sequences defined 
by $n_{k+1}^i:=\lfloor \alpha_i n_k^i + \beta_i \ln n_k^i \rfloor$, where $\lfloor \cdot \rfloor$ is the floor function
and $n_0^i=n_0$ is chosen large enough that both sequences are strictly increasing,
and let $\sigma_k^i$ denote their phase sequences. We let $U_0:=U \cap \{z=0\}$, and we assume without loss of generality
that $U_0$ is a disk centered at $w=0$.

In what follows we write $q_i(w):=\proj_2\circ P_i(0,w)$ for $i\in \{1,2\}$.

\begin{lem}\label{lem:insideU}
Let $z \in \bcal_{p_1}$, $w \in \bcal_{q_1} \cap U_0$, and for any $n \in \N$, let $z_n:=p_1^n(z)$. 
	Then there exists $m \in \N$ such that for all $k$ large enough and all $0 \leq j \leq n_{k+1}^1-n_k^1-m$, $P_1^j(z_{n_k^1},w) \in U$.
\end{lem}

\begin{proof}
	First, note that since $\limn z_n = 0$, $z_{j+n_k}$ belongs to any arbitrary neighborhood for 
	$j \geq 0$ and $k$ large enough. Therefore, if we let $w_j$ denote the second component of 
	$P_1^j(z_{n_k^1},w)$, it is enough to prove that for $k$ and $m$ large enough, $w_j$ remains in $U_0$
	for all  $0 \leq j \leq n_{k+1}^1-n_k^1-m$. 
	For $0 \leq j \leq t_k:=\lfloor (n_k^1)^\nu \rfloor$, this follows from Lemma \ref{lem:init}.
	For $t_k \leq j \leq n_{k+1}^1-n_k^1-\lfloor \alpha_1 t_k \rfloor$
	, this follows from 
	Lemma \ref{prop:eggb} (recall that $\nu$ is a fixed constant in $(\frac{1}{2},\frac{2}{3})$).
	Finally, the existence of $m>0$ (independant from $n_k^1$) such that for all 
	$ n_{k+1}^1-n_k^1-\lfloor \alpha_1 t_k \rfloor \leq j \leq  n_{k+1}^1-n_k^1 - m$ 
	we have $w_j \in U_0$
	follows from Lemma \ref{lem:exit}.
\end{proof}

Let us now prove Proposition \ref{coro:mn1=mn2}.

\begin{prop}
	The real numbers $\alpha$ and $\beta$ are topological invariants, i.e. $\alpha_1=\alpha_2$ and $\beta_1=\beta_2$.
\end{prop}

\begin{proof}
	Let $z \in \bcal_p \cap \mathrm{Dom}(\mathfrak{f})$ and $w \in \bcal_{q_1} \cap U_0$.
	By Lemma \ref{lem:insideU}, we have
	\begin{equation}\label{eq:conj}
	\mathfrak h \circ P_1^{j}(p_1^{n_k^1}(z),w)=P_2^j \circ \mathfrak h(p_1^{n_k^1}(z),w)
	\end{equation}
	for all $0 \leq j \leq n_{k+1}^1 - n_k^1-m$.
	In particular, both sides of the equation belong to 
	$V$.

	Let $M_k:=\lfloor (\alpha_2-1) n_k^1 + \beta_2 \ln n_k^1 \rfloor$, and let 
	$\rho_k:=\{ (\alpha_2-1) n_k^1 + \beta_2 \ln n_k^1 \}$.
	Chose $R>0$ large enough that $V \subset \D(0,R)^2$, and 
	choose $(z,w) \in U$ so that 
	$$\left|\lcal_2(\alpha_2, \Gamma_2 -\rho_k; \mathfrak{f}(z),\mathfrak{g}_0(w))\right|>R$$
	 for arbitrarily large values of $k$.
	This is always possible: indeed,  let $\rho \in [0,1)$ be an accumulation point
	of the sequence $s_k$. From the functional equation
	$$\lcal_2(\alpha_2, \Gamma_2-\rho; p_2(z), q_2(w))=q_2 \circ \lcal_2(\alpha_2, \Gamma_2-\rho; z,w),$$ it follows that 
	$(z,w) \mapsto \lcal_2(\alpha_2,\Gamma_2- \rho; z,w)$ takes arbitrarily large values on $U$. Then  any $(z,w) \in U$ such that $\left|\lcal_2(\alpha_2,\Gamma_2- \rho; z,w) \right|>R$ works.

	Next, we observe that it follows from the Theorem \ref{th:maintech} that 
	\begin{align*}
		P_2^{M_k} \circ \mathfrak h(p_1^{n_k^1}(z),w)&= P_2^{M_k} \left(p_2^{n_k^1}(\mathfrak f(z)), \mathfrak g_0(w)+o(1) \right) \\
		&=\left(0, \lcal_2(\alpha_2,\Gamma_2-\rho_k;\mathfrak f(z),\mathfrak{g}_0(w)\right) + o(1).
	\end{align*}
 	Therefore, by \eqref{eq:conj} and our choice of $R, z$  and $w$, we must have 
 	$M_k > n_{k+1}^1-n_k^1-m$  for arbitrarily large values of $k$. Therefore $\alpha_2 \geq \alpha_1$; but then by symmetry, $\alpha_2=\alpha_1$.
 	Then, using again the fact that $M_k > n_{k+1}^1-n_k^1-m$, we find $\beta_2 \geq \beta_1$, and therefore 
 	we finally have, again by symmetry, $\beta_1=\beta_2$.
%
%
%
%
%
\end{proof}

We are now ready to prove Theorem \ref{thm:top}.

\begin{proof}[Proof of Theorem \ref{thm:top}]
	By Proposition \ref{coro:mn1=mn2}, we have $n_k^1=n_k^2=:n_k$, $\sigma_k^1=\sigma_k^2=:\sigma_k$ and $\alpha_1=\alpha_2=:\alpha$, so
	by \eqref{eq:conj} and the Main Theorem, we have
	\begin{equation*}
	\mathfrak h\left( o(1), \lcal_1(\alpha,\Gamma_1 + \sigma_k-m; z,w ) + o(1)   \right)
	= \left( o(1),  \lcal_2(\alpha,\Gamma_2  + \sigma_k-m; \mathfrak{f}(z),\mathfrak{g}_0(w)  \right) + o(1)
	\end{equation*}
	for all $k$ large enough, $z \in B_p$, and $w \in U_0 \cap B_{q_0}$.
	Therefore, for any accumulation point $\sigma$ of the sequence $(\sigma_k)_{k \geq 0}$, we have:
	\begin{equation}\label{eq:conj2}
	\mathfrak{g}_0\left( \lcal_1(\alpha, \Gamma_1 + \sigma - m; z,w)\right) = \lcal_2(\alpha,\Gamma_2 + \sigma - m; \mathfrak{f}(z), \mathfrak{g}_0(w))
	\end{equation}
	Let us write for simplicity $\lcal_i(z,w) := \lcal_i(\alpha,\Gamma_i+\sigma-m;z,w)$.
	Observe that since $\mathfrak{f}$ and $\mathfrak{g}_0$ conjugate $p_1$ to $p_2$ and $q_1$ to $q_2$ respectively, there exists 
	homeomorphisms $\tilde{\mathfrak f} : \C \to \C$ and $\tilde{\mathfrak g}_0: \C \to \C$ commuting with the translation by 1 such that:
	\begin{equation}
	\tilde{\mathfrak{g}}_0\circ  \phi_{1}^o=\phi_{2}^o\circ \mathfrak{g}_0 ,
	\end{equation}
	and 
	\begin{equation}
	\tilde{\mathfrak{f}}\circ  \phi_{p_1}^\iota=\phi_{p_2}^\iota\circ \mathfrak{f} ,
	\end{equation}
	
where $\phi_i^o$ denotes the outgoing Fatou coordinate of $q_i$.

	For $z,w$ as above, let $Z:=\phi_{p_1}^\iota(z)$ and $W:=\phi_1^\iota(w)$.
	Let us compute: 
	\begin{align*}
	\tilde{\mathfrak{g}}_0 \circ \tilde \horn^1_{Z,\sigma_1}(W) &= \phi_2^o \circ \mathfrak{g}_0 \circ (\phi_1^o)^{-1} \circ \tilde \horn^1_{Z,\sigma_1}(W)\\
	\text{(by \eqref{eq:hornlav})}\quad	 &=\phi_2^o \circ \mathfrak{g}_0 \circ \lcal_1(z, (\phi_1^o)^{-1}(W)) \\
	\text{(by \eqref{eq:conj2})}\quad	 &=\phi_2^o \circ \lcal_2(\mathfrak{f}(z),\mathfrak{g}_0 \circ (\phi_1^o)^{-1}(W)) \\
	&=\alpha \phi_2^\iota \circ \mathfrak{g}_0 \circ (\phi_1^o)^{-1}(W) + (1-\alpha) \phi_{p_2}^\iota(\mathfrak{f}(z)) + \sigma_2\\
	&=\alpha \phi_2^\iota \circ (\phi_2^o)^{-1} \circ \tilde{\mathfrak{g}}_0(W) + (1-\alpha) \tilde{\mathfrak{f}}(Z) + \sigma_2 \\
	&=\tilde \horn^2_{\tilde{\mathfrak{f}}(Z),\sigma_2} (\tilde{\mathfrak{g}}_0(W))
	\end{align*}
	where $\sigma_i =  \sigma + \Gamma_i-m$.
	
	Therefore, if we let $G(Z,W)=(\tilde{\mathfrak{f}}(Z), \tilde{\mathfrak{g}}_0(W))$ we have proved that 
	$$G \circ \tilde \horn_{\sigma_1}^1(Z,W) = \tilde \horn_{\sigma_2}^2 \circ G(Z,W).$$
	
	This relation holds for all $z \in \bcal_{p_1}$ and for all $w \in \bcal_{q_1} \cap U_0$; 
	therefore it holds for all $Z \in \C$ and all $W \in \C$ with $\re (W)$ large enough.
	
	But since the lifted horn maps $\tilde \horn_{\sigma_i}^i$ commute with the translation of vector 
	$(1,1)$, this conjugation descends to a conjugation of the horn maps on $\C^2/\Z$.
\end{proof}

\begin{coro}	
	If $P_1$ and $P_2$ are topologically conjugated near $(0,0)$, then the number of critical points 
	of $q_i$ in $\bcal_{q_i}$ is the same. In particular, there is no $k \in \N$ such that the local topological conjugacy class of maps of the form \eqref{map1} depend only the $k$-jet of $P$ at the origin.
\end{coro}

\begin{proof}
	For any $Z \in \C$, the number of critical values of $\tilde \horn_{Z,\sigma}$
	in $\{0< \re\, W \leq 1 \}$ is exactly equal to the number of critical points of $q_0$ in $\bcal_{q_0}$.
	The former is clearly preserved under the topological conjugacy $G$, therefore so is the latter.
	
	For the second assertion, it suffices to observe that this number cannot depend on
	any $k$-jet of $q_0$ at $w=0$.
\end{proof}

\section{Proof of Corollary  \ref{coro:hist}}\label{sec:proofhist}

Finally, we will prove Corollary \ref{coro:hist} in this Section.

Let $\lcal^{(i)}$ denote the extended Lavaurs maps associated to both parabolic fixed points $(0,w_i)$, 
and let $\lcal_z^{(i)}(w):=\lcal(\alpha_i,\Gamma_i; z,w)$.
Let $\mcal_z:=\lcal_{z}^{(2)} \circ \lcal_{z}^{(1)}$.
We denote by $\bcal_i$ the parabolic basins of $w_i$ for $q_0$, so that 
$(z,w) \mapsto \mcal_z(w)$ is defined on $\bcal_p \times \bcal_1$.
We start by recalling the notion of islands, named after Alhlfors famous Five Islands theorem.

\begin{defi}
	Let $f: U \to \rs$ be a holomorphic map, where $U \subset \rs$ is a domain. Let $D_1 \subset \rs$ 
	be a Jordan domain. We say that $D$ is an island for $f$ over $D_1$ if $f: D \to D_1$ is a conformal isomorphism.
\end{defi}

\begin{lem}\label{lem:islands}
	Let $f(z)=z+z^2+O(z^3)$ be a polynomial map with a parabolic fixed point, and let 
	$\phi_f^\iota: \bcal_f \to \C$ and $\psi_f^o: \C \to \C$ denote its incoming Fatou coordinate
	and outgoing Fatou parametrization respectively. Then 
	\begin{enumerate}
		\item For every Jordan domain $D_1 \subset \C$ such that $(\phi_f^\iota)^{-1}(D_1)$ doesn't intersect critical orbits of $f$, for every open set $\Omega$ intersecting $\partial \bcal_f$, $\phi_f^\iota$ has an island $D \Subset \Omega$ over $D_1$
		\item For every Jordan domain $D_1 \subset \C$ that doesn't intersect the postcritical set of $f$, 
		$\psi_f^o$ has an island $D$ over $D_1$.
	\end{enumerate}
\end{lem}

%
%
%

\begin{proof}
	Let $D_0 \subset \C$ be a Jordan domain, and let $\Omega$ be an open set intersecting 
	$\partial \bcal_f$.	
	
	Let $D_{k}:=D_0+k$. It is well-known that $\phi_f^\iota: \bcal_f \to \C$ is a branched cover
	whose critical points are the pre-critical orbits of $f$ in $\bcal_f$;
	therefore, by the assumptions on $D_0$, $D_k$ is simply connected and doesn't contain any critical value of $\phi_f^\iota$, so  $\phi_f^\iota$ has an island $U_0$ above $D_k$.
	
	By assumption, $U_0$  doesn't meet any critical orbits of $f$, and it is simply connected, so we may define univalent inverses branches of $f^{-k}$ for all $k$, and for $k$ large enough, at least one such branch $g_k$ of $f^{-k}$ will map $U_0$ compactly into $\Omega$ (by normality and the equidistribution of preimages).
	Let $U_k:=g_k(U_0)$. We then have:
	\begin{align*}
		\phi_f^\iota(U_0) = \phi_f^\iota \circ f^k(U_k) =  \phi_f^\iota(U_k)+k = D_0+k
	\end{align*}	
	so that $\phi_f^\iota(U_k) = D_0$. The domain $U_k$ is the desired island above $D_0$.

	The second item follows immediately from the other well-known fact that 
	$$
	\psi_f^o: \C \backslash (\psi_f^o)^{-1}(P_f) \to \C \backslash P_f
	$$
	 is a covering map, 	where $P_f$ denotes the post-critical set of $f$. 
	
\end{proof}

\begin{lem}\label{lem:fixpm}
	There exists $z_0 \in \bcal_p$ such that $\mcal_{z_0}$ has a super-attracting fixed point $w_0$.
\end{lem}

\begin{proof}
	The difficulty is that we cannot apply Montel's theorem, as the domain of $\mcal_z^n$
	shrinks as $n \to +\infty$. Instead, we will follow closely  the proof  of the Shooting Lemma from \cite{astorg2021bifurcation}.
	Let $\phi_i^\iota$ (with $i=1,2$) denote the incoming Fatou coordinates of $w_i$ for $q_0$,
	and let $\psi_i$ denote the outgoing Fatou parametrizations associated to $w_i$ for $q_0$.
	Let $Z:=\phi_p^\iota(z)$, $A_{i,Z}(W):=\alpha_0 W + (1-\alpha_0) Z+\Gamma_i$, so that 
	\begin{equation}
		\mcal_z = \psi_2 \circ A_{2,Z} \circ \phi_2^\iota \circ \psi_1 \circ A_{1,Z} \circ \phi_1^\iota.
	\end{equation}
	
	Let $c \in \bcal_{1}$ be a critical point for $\phi_1^\iota$, and let $x \in \psi_2^{-1}(\{c\})$.
	Let $\gamma(Z):=A_{1,Z} \circ \phi_1^\iota(c)$, and let 
	$g_Z:=A_{2,Z} \circ \phi_2^\iota \circ \psi_1$. If we can find $Z \in \C$ such that 
	$g_Z \circ\gamma(Z)=x$, then this will mean that $\mcal_z(c)=c$, where $\phi_p^\iota(z)=Z$,
	which will prove the Lemma.
	
	Let $U_0:= \psi_1^{-1}(\bcal_{2})$. Since $\psi_i: \C \to \C$ is entire, $U_0 \subset \C$ is an open set with non-empty boundary. 
	From the expression of $\gamma$, if we fix any $W_0 \in \partial U_0$, we can find explicitly some $Z_0 \in \C$
	such that $ \gamma(Z_0)=W_0$. 
	
	Let us observe that $g_{Z}=g_{Z_0}+(1-\alpha_2)(Z-Z_0)$. Therefore, letting $h(Z):=x+(\alpha_2-1)(Z-Z_0)$, the equation $g_Z \circ \gamma(Z)=x $ is equivalent to
	\begin{equation}\label{eq:gg=h}
		 g_{Z_0} \circ \gamma(Z) = h(Z).
	\end{equation}

	Let $D$ be a disk centered at $x$ such that $D$ contains no critical values of $g_{Z_0}$. (This is possible because the set of critical values of $g_Z$ is discrete, in fact finite in $\C/\Z$, and we may assume that $x$ is not one of them). Let $\eps>0$ be small enough that 
	$h(\D(Z_0,\eps)) \Subset D$. Let $\Omega:=\gamma(\D(Z_0,\eps))$: $\Omega$ is an open neighborhood of $W_0 \in \partial U_0$. By Lemma \ref{lem:islands}, there exists $D_1 \Subset \Omega \cap U_0$
	such that $g_{Z_0}: D_1 \to D$ is a conformal isomorphism. 
	In particular, $g_{Z_0} \circ \gamma: V \to D$ is a conformal isomorphism, 
	where $V:=\gamma^{-1}(D_1)$ is a disk that is compactly contained in $\D(Z_0,\eps)$.
	By the definition of $\eps$ and $V$, we therefore have $h(V) \Subset g_{Z_0} \circ \gamma(V)=D$, 
	and $D,V$ are disks with smooth boundaries. It then follows from the Argument Principle 	that there exists $Z \in V$ satisfying \eqref{eq:gg=h}, and the Lemma is proved.
	
\end{proof}

\begin{proof}[Proof of Corollary \ref{coro:hist}]
	We consider an inductive sequence of integers 
	defined by $n_{k+1}=\alpha_1 n_k$ if $k$ is even and $n_{k+1}=\alpha_2 n_k$ if $k$ is odd.
	
	By the Main Theorem applied twice, we have 
	$$P^{n_{k+2}-n_k}(z_{n_k},w) = (z_{n_{k+2}}, \mcal_z(w)) + o(1)$$
	with local uniform convergence for $(z,w)$ sufficiently close to the point $(z_0,w_0)$ given by  Lemma \ref{lem:fixpm}.

	Since $w_0$ is a super-attracting fixed point for $\mcal_{z_0}$, there exists 
	$r>0$ such that $\mcal_{z_0}(\D(w_0,r)) \Subset \D(w_0,\frac{r}{2})$, and by continuity there exists 
	$\eta>0$ such that for all $z \in \D(z_0,\eta)$ we have $\mcal_z(\D(w_0,r)) \Subset \D(w_0,r)$.
	
	Let $V$ be a connected component of $P^{-n_0}(p^{n_0}(\D(z_0,\eta)) \times \D(w_0,r))$.
	For $n_0$ large enough and $(n_k)$ satisfying the induction relation above, we have, for any $k \in \N$ and $(z,w) \in V$:
	\begin{equation}
		P^{n_{2k}}(z,w)\in \bcal_p \times \D(w_0,r).
	\end{equation} 
	
	In particular, $V \subset K(P)$, and therefore 
	$V$ is contained in the Fatou set of $P$. Let $\Omega$ be the Fatou component of $P$ containing $V$.
	

	Finally, let us prove that $\Omega$ satisfies the historicity property. Observe that
	\begin{equation}\label{eq:hist1}
		\lim_{k \to +\infty} \frac{1}{n_{2k+1}-n_{2k}} \sum_{j=n_{2k}}^{n_{2k+1}} \delta_{P^j(z,w)} 
		= (0,w_1)
	\end{equation}
	and 
	\begin{equation}\label{eq:hist2}
		\lim_{k \to +\infty} \frac{1}{n_{2k+2}-n_{2k+1}} \sum_{j=n_{2k+1}}^{n_{2k+2}} \delta_{P^j(z,w)} 
		= (0,w_2).
	\end{equation}
	This follows from the fact that it takes $n_{2k+1}-n_{2k}$ 
	iterations to "pass through the eggbeater" associated to $(0,w_1)$, and $n_{2k+2}-n_{2k}$ to pass through the one associated to $(0,w_2)$ (more precisely, this follows from 
	Lemma \ref{prop:eggb}). 
	Let $(z,w) \in V$, and let us consider $e_n=e_n(z,w):=\frac{1}{n} \sum_{j=0}^n \delta_{P^j(z,w)}$.

	By \eqref{eq:hist1}, we have
	\begin{align*}
		e_{n_{2k+1}}&=e_{n_{2k}} \frac{n_{2k}}{n_{2k+1}} + (1-\frac{n_{2k}}{n_{2k+1}}) \delta_{(0,w_1)}+o(1) \\
		&=\frac{1}{\alpha_1} e_{n_{2k}} + (1-\frac{1}{\alpha_1}) \delta_{(0,w_1)}+o(1) 
	\end{align*}
	
	and similarly, using \eqref{eq:hist2}, 	
	\begin{equation*}
		e_{n_{2k}}=\frac{1}{\alpha_2} e_{n_{2k-1}} + (1-\frac{1}{\alpha_2}) \delta_{(0,w_2)}+o(1).
	\end{equation*}

	Putting last two equations together, we find:
	\begin{equation}
		e_{2k}=\frac{\alpha_1 \alpha_2-\alpha_2}{\alpha_1 \alpha_2-1} \delta_{(0,w_1)}
		+ \frac{\alpha_2-1}{\alpha_1 \alpha_2-1} \delta_{(0,w_2)} + o(1)
	\end{equation}
	and
	\begin{equation}
		e_{2k+1}=  \frac{\alpha_1-1}{\alpha_1 \alpha_2-1} \delta_{(0,w_1)}  +
		\frac{\alpha_1 \alpha_2-\alpha_1}{\alpha_1 \alpha_2-1} \delta_{(0,w_2)}+ o(1).
	\end{equation} 
\end{proof}

\bibliographystyle{amsplain}
\bibliography{bibliography}

\end{document}